\documentclass[11pt,reqno]{amsart}

\usepackage{amsmath, amsfonts, amsthm, amssymb,mathrsfs, amscd, graphicx, cite, color}

\newtheorem{thm}{Theorem}[section]

\newtheorem{lemma}{Lemma}[section]

\theoremstyle{definition}

\theoremstyle{remark}
\newtheorem{rem}{Remark}[section]

\numberwithin{equation}{section}

\allowdisplaybreaks

\newcommand{\R}{{\mathbb R}}

\def\f{\frac}

\def\hf1{^\f{1}{1-\xi^2}}

\def\be{\begin{equation}}
\def\en{\end{equation}}
\def\bs{\begin{split}}
\def\es{\end{split}}

\author{Dehua Wang  and Zhuan Ye }
\address{Department of Mathematics, University of Pittsburgh,
                           Pittsburgh, PA 15260, USA}
\email{dwang@math.pitt.edu}

\address{Department of Mathematics and Statistics, Jiangsu Normal University,  Xuzhou, Jiangsu 221116, P. R. China}
\email{yezhuan815@126.com}

\title[Global existence and exponential decay of Navier-Stokes equations]
{Global existence and exponential decay of strong
solutions for the inhomogeneous incompressible Navier-Stokes equations with vacuum}

\date{\today}

\begin{document}

\begin{abstract}
The inhomogeneous incompressible Navier-Stokes equations with fractional Laplacian dissipations in the multi-dimensional whole space are considered. The existence and uniqueness of global strong solution with vacuum are established for large initial data. The  exponential decay-in-time of the   strong solution is also obtained, which is different from the homogeneous case. The initial density may have vacuum and even compact support.
\end{abstract}
\maketitle

\section{Introduction}

In this paper, we are concerned with the Cauchy problem of the following fractional inhomogeneous incompressible
Navier-Stokes equations:
\begin{equation}\label{INDNSE}
\left\{\begin{array}{l}
\partial_t \rho+{\rm div}(\rho u)
       =0 ,\vspace{2mm} \qquad   x\in \mathbb{R}^{n},\,t>0,\\
\partial_t(\rho u) + {\rm div}(\rho u\otimes u) +\mu(-\Delta)^{\alpha}u+\nabla p =0,
             \vspace{2mm}\\
\nabla\cdot u=0,\vspace{2mm}\\
\rho(x,0)=\rho_{0}(x),\quad u(x,0)=u_{0}(x),
\end{array}\right.
\end{equation}
where $\rho=\rho(x,t)$ denotes the density, $u=u(x,t)=(u_{1}(x,t),u_{2}(x,t),\cdot\cdot\cdot,u_{n}(x,t))$ denotes the fluid velocity, $p(x,t)$ is the scalar pressure, and  $\mu>0$ is the   viscosity that is assumed to be one for simplicity; $\rho_{0}(x)$ and $u_{0}(x)$
are the prescribed initial data for the density and   velocity with   $\nabla\cdot u_{0}=0$.
The fractional Laplacian operator $(-\Delta)^{\alpha}$ with $\alpha>0$ is defined via the Fourier transform as
$$\widehat{(-\Delta)^{\alpha} f} (\xi) = |\xi|^{2\alpha}\, \widehat{f}(\xi),$$
where $\widehat{f}$ is the Fourier transform of $f$.
Recently there have been a lot of studies on the fractional Laplace-type problems, not only for   mathematical interests but also for various applications in different fields. As a matter of fact, the application
background of the fractional problems can be found in fractional quantum mechanics \cite{Laskin}, probability \cite{Applebaum,Bertoinl}, overdriven detonations in gases \cite{Clavind}, anomalous diffusion in semiconductor growth \cite{Woyczyh}, physics and chemistry \cite{Metzlerk}, optimization and finance \cite{ConTankov} and so on.

When $\alpha=1$, the system \eqref{INDNSE} becomes the classical inhomogeneous incompressible Navier-Stokes equations,  describing    fluids    inhomogeneous in density. Typical examples of such fluids include the mixture of incompressible and non-reactant flows, flows with complex structure (e.g. blood flows or rivers), fluids containing a melted substance, etc. We refer to \cite{Li1998} for the detailed derivation of this system.
Because of its physical importance, complexity, rich phenomena and mathematical challenges, there is a notablly large literature on the mathematical studies on the   well-posedness of solutions to the classical inhomogeneous incompressible Navier-Stokes equations. For example, when the initial density is strictly positive, Kazhikov \cite{Ka} proved that the system has at least one global weak solution in the energy space,  the local (global if $n=2$) existence and uniqueness of strong solutions
were first obtained in \cite{AKM,LS}, and similar results were established recently in a series of works such as  \cite{Abidi1,Abidi2,Cheminpp,Danchin1prse,Danchin,Paicu1,Paicu2,
DanchinMucha,DanchinMucha1}.
For the initial data with vacuum, the problem becomes much more complicated due to the possible degeneracy near vacuum.
Simon \cite{Si1990} first proved the global existence of weak solutions with finite energy, which was  extended later by Lions \cite{Li1998} to the case of density-dependent viscosity.
For the strong solution in dimensions three, Choe-Kim \cite{CK2003} proposed a compatibility condition and successfully established the local existence of strong solutions, which was  improved by Craig-Huang-Wang \cite{Craighw} for global strong small solutions (see \cite{HW2015jde,Zhang15jde,Hellsd} for the case of density-dependent viscosity). In the case $n=2$, it was shown in \cite{HW2013,LsZhong} that the initial-boundary value problem and the Cauchy problem of the inhomogeneous Navier-Stokes equations with vacuum admits a unique global strong solution for the general initial data.
For $n\ge 3$, the global existence of strong or smooth solutions with general initial data is a well-known open problem.
One  difficulty is that the Laplacian dissipation is insufficient to control the nonlinearity when applying the standard techniques to establish global a priori bounds. Hence it is natural to explore the problem via replacing the Laplacian operator by the fractional Laplacian operators as in \eqref{INDNSE}, motivated by the applications aforementioned, in order to obtain the global strong solution for the general initial data, which is the aim of this paper.

When the density $\rho$ is a constant, the system (\ref{INDNSE}) becomes the classical fractional homogeneous incompressible Navier-Stokes equations, which admit a unique global smooth solution as long as $\alpha\geq\frac{1}{2}+\frac{n}{4}$.
This result dates back to J. Lions's book \cite{Lions} in 1969, which is even true for some logarithmic corrections (see \cite{ttTao,Barbatomr} for details). These results were extended to the inhomogeneous system \eqref{INDNSE} in \cite{Fangz} for $\alpha\geq\frac{1}{2}+\frac{n}{4}$ and in \cite{hanw} for the corresponding logarithmic case.
It should be noted that both \cite{Fangz} and \cite{hanw} require   the initial density $\rho_{0}$  bounded away from zero, i.e., the flow has no vacuum.
The  goal of this paper is to relax this restriction. More precisely, we shall establish the global existence of strong solutions with
vacuum to the system \eqref{INDNSE}. Moreover, we shall also obtain the exponential decay-in-time of the   strong solution.
We recall that $(\rho,u)$ is called a weak solution to the system \eqref{INDNSE} if it satisfies \eqref{INDNSE} in the sense of distributions, and a strong solution if the system \eqref{INDNSE} holds almost everywhere.

In this paper, we shall  adopt  the convention that $C$ denotes a generic constant depending only on the initial data.
For simplicity, we will frequently use the notation $\Lambda:=(-\Delta)^{\frac{1}{2}}$.
For $1\leq r\leq\infty$ and integer $k\geq0$, we use the following notations for the standard homogeneous and inhomogeneous Sobolev spaces:
$$L^{r}=L^{r}(\mathbb{R}^{n}),\quad \dot{W}^{k,r}=\{g\in L_{loc}^{1}(\mathbb{R}^{n}):\ \|g\|_{\dot{W}^{k,r}}:=\|\nabla^{k}g\|_{L^{r}}<\infty\},\ \  W^{k,r}:=L^{r}\cap \dot{W}^{k,r},$$
$$\dot{H}^{s}=\left\{g:\ \|g\|_{\dot{H}^{s}}^{2}=\int_{\mathbb{R}^{n}}|\xi|^{2s}|\widehat{f}(\xi)|^{2}\,d\xi
<\infty\right\},\quad  {H}^{s}:=L^{2}\cap \dot{H}^{s}.$$

Now we state our main result of this paper as follows.

\begin{thm}\label{Th1}
For the system (\ref{INDNSE}) with $\alpha=
\frac{1}{2}+\frac{n}{4}$ and $n\geq3$, if the initial data $(\rho_{0},\,u_{0})$
satisfies the following conditions:
$$0\leq\rho_{0}\in  L^{\frac{2n}{n+2}}(\mathbb{R}^{n})\cap L^{\infty}(\mathbb{R}^{n}),\quad \nabla\rho_{0}\in L^{\frac{4n}{n+6}}(\mathbb{R}^{n})\cap L^{2}(\mathbb{R}^{n}),$$
$$ \nabla\cdot u_{0}=0,\quad u_{0}\in \dot{H}^{\frac{1}{2}+\frac{n}{4}}(\mathbb{R}^{n}),\quad \sqrt{\rho_{0}}u_{0}\in L^{2}(\mathbb{R}^{n}),$$
then it has a unique global strong solution $(\rho,u)$ such that,  for any given $T>0$ and for any $0<\tau<T$,
$$0\leq\rho\in L^{\infty}(0,T; L^{\frac{2n}{n+2}}(\mathbb{R}^{n})\cap L^{\infty}(\mathbb{R}^{n})),\quad \nabla\rho\in L^{\infty}(0,T; L^{\frac{4n}{n+6}}(\mathbb{R}^{n})\cap L^{2}(\mathbb{R}^{n})),$$
$$u \in L^{\infty}(0,T; \dot{H}^{\frac{1}{2}+\frac{n}{4}}(\mathbb{R}^{n}))\cap L^{2}(0,T; \dot{H}^{1+\frac{n}{2}}(\mathbb{R}^{n})),
\quad \sqrt{\rho}\partial_{t}u\in L^{\infty}(\tau,T; L^{2}(\mathbb{R}^{n})),$$
$$ \sqrt{\rho}\partial_{tt}u\in L^{2}(\tau,T; L^{2}(\mathbb{R}^{n})),\quad
\partial_{t}u \in L^{2}(\tau,T; \dot{H}^{\frac{1}{2}+\frac{n}{4}}(\mathbb{R}^{n}))\cap L^{\infty}(\tau,T; \dot{H}^{\frac{1}{2}+\frac{n}{4}}(\mathbb{R}^{n})),$$
$$\Lambda^{1+\frac{n}{2}}u\in L^{\infty}(\tau,T;L^{\frac{4n}{n-2}}(\mathbb{R}^{n})) ,\quad    p \in L^{\infty}(\tau,T; {H}^{1}(\mathbb{R}^{n}))\cap \in L^{\infty}(\tau,T; {W}^{1,\frac{4n}{n-2}}(\mathbb{R}^{n})).$$
Moreover, there exists some positive constant $\gamma$ depending only on $\|\rho_{0}\|_{L^{\frac{2n}{n+2}}}$ such that,  for all $t\geq1$,
\begin{eqnarray}\|\Lambda^{\frac{1}{2}+\frac{n}{4}}u(t)\|_{L^{2}}^{2} &+&\|\sqrt{\rho}\partial_{t}u(t)\|_{L^{2}}^{2}
+\|\Lambda^{1+\frac{n}{2}}u(t)\|_{L^{2}\cap L^{\frac{4n}{n-2}}}^{2} +\|\Lambda^{\frac{1}{2}
+\frac{n}{4}}\partial_{t}u(t)\|_{L^{2}}^{2}\nonumber\\&&+\|p(t)\|_{H^{1}\cap W^{1,\frac{4n}{n-2}}}^{2}
\leq \widetilde{C}e^{-\gamma t},\nonumber
\end{eqnarray}
where $\widetilde{C}$ depends only on $\|\rho_{0}\|_{L^{\frac{2n}{n+2}}}$, $\|\rho_{0}\|_{L^{\infty}}$, $\|\nabla\rho_{0}\|_{L^{2}}$, $\|\sqrt{\rho_{0}}u_{0}\|_{L^{2}}$ and $\|\Lambda^{\frac{1}{2}+\frac{n}{4}}u_{0}\|_{L^{2}}$.
\end{thm}

\begin{rem}
For  the exponential decay-in-time property of Theorem \ref{Th1},  the estimate of the density:
\begin{align}\label{density23}
\|\rho(t)\|_{L^{\frac{2n}{n+2}}}\leq \|\rho_{0}\|_{L^{\frac{2n}{n+2}}}
\end{align}
plays a crucial role. This estimate \eqref{density23} does not hold for the homogeneous case (with constant density) in the whole space. In fact, only algebraic decay rate has been obtained for the homogeneous case in literature, e.g., \cite{Abidi2,Chenzmcpde,JiuY15,kato84,Wiegner,Schonbek}.
\end{rem}

\begin{rem}
As a consequence of the proof of Theorem \ref{Th1}, the corresponding conclusions of the global existence and exponential decay of strong solutions are also valid  for the system (\ref{INDNSE}) with at least $\frac{1}{2}+\frac{n}{4}<\alpha<\frac{n}{2}$. We also remark that our arguments can be adopted to other similar systems with the same dissipations, such as the inhomogeneous incompressible magnetohydrodynamic equations.
\end{rem}

\begin{rem}
Under the assumption that the initial velocity is suitably small,
the exponential decay-in-time of the strong solutions was obtained in \cite{Hellsd} for the Cauchy problem of the three-dimensional  classical inhomogeneous incompressible Navier-Stokes equations (i.e., the system (\ref{INDNSE}) with $\alpha=1$) with density-dependent viscosity and vacuum, which of course is valid for the constant viscosity case.
We remark that Theorem \ref{Th1} is proved without any smallness on the initial data. Moreover, the initial density is allowed to have  vacuum. We also point out that the regularity assumption on the initial density $\nabla\rho_{0}\in L^{\frac{4n}{n+6}}$ is used only to guarantee the uniqueness of the solution.
\end{rem}

\begin{rem}
Finally, compared with the previous works \cite{CK2003,Craighw,HW2013,HW2015jde,Zhang15jde}, the following corresponding compatibility condition on the initial data is dropped from Theorem \ref{Th1}:
\begin{equation}\label{compatibility}
(-\Delta)^{\frac{1}{2}+\frac{n}{4}}u_{0}+\nabla p_{0}=\sqrt{\rho_{0}}g,
\end{equation}
with $(p_{0},\,g)\in  {H}^{1}(\mathbb{R}^{n})\times L^{2}(\mathbb{R}^{n})$.
However, without the compatibility condition, the price that we need to pay is that the parameter $\tau$ in Theorem \ref{Th1} must be positive and can not be replaced by  the initial time $\tau=0$.
\end{rem}

We now outline the main idea and make some comments on the proof of this theorem.
The local existence of strong solutions to the system (\ref{INDNSE}) follows from the works in literature such as \cite{CK2003,lijink}  (see Lemma \ref{LT}). Thus our efforts are devoted to establishing global a priori estimates on strong solutions to the system (\ref{INDNSE}) in suitable higher-order norms.
It should be pointed out that compared with  the related  works in literature, the proof of Theorem \ref{Th1} is much more involved due to the absence of the positive lower bound for the initial density as well as the absence of the smallness and the compatibility conditions for the initial velocity. Consequently, some new ideas are needed to overcome these difficulties as explained below.
First, taking the advantage of the estimate \eqref{density23} on the density, 
we have the following key observation:
\begin{eqnarray}
\|\sqrt{\rho}u\|_{L^{2}}\leq \|\sqrt{\rho}\|_{L^{\frac{4n}{n+2}}}\|u\|_{L^{\frac{4n}{n-2}}}\leq C \|{\rho}\|_{L^{\frac{2n}{n+2}}}^{\frac{1}{2}}
\|\Lambda^{\frac{1}{2}+\frac{n}{4}}u\|_{L^{2}}\leq C
\|\Lambda^{\frac{1}{2}+\frac{n}{4}}u\|_{L^{2}},\nonumber
\end{eqnarray}
which implies that $\|\sqrt{\rho}u(t)\|_{L^{2}}^{2}$ decays with the rate of $e^{-\gamma t}$ for some $\gamma>0$ depending only on $\|\rho_{0}\|_{L^{\frac{2n}{n+2}}}$ (see Lemma \ref{INDL221} for details). With the help of this key exponential decay-in-time rate, we can show that $\|\Lambda^{\frac{1}{2}+\frac{n}{4}}u(t)\|_{L^{2}}^{2}$ decays at the same
rate as $e^{-\gamma t}$ (see Lemma \ref{INDL222} for details). The   next step is to derive the bound of $\|\sqrt{\rho}\partial_{t}u(t)\|_{L^{2}}^{2}$. However, it prevents us to achieve this goal due to the absence of the compatibility condition \eqref{compatibility} for the initial velocity. To overcome this difficulty, we first derive the following crucial time-weighted estimate (see (\ref{addok5251})):
\begin{eqnarray}\label{testiamt11}
t\|\sqrt{\rho}\partial_{t}u(t)\|_{L^{2}}^{2} +\int_{0}^{t}{\tau\|\Lambda^{\frac{1}{2}+\frac{n}{4}}\partial_{\tau}u(\tau)
\|_{L^{2}}^{2}\,d\tau} \leq C,\quad \forall\,t\geq0,\end{eqnarray}
where the positive constant $C$ is independent of the initial data of $\sqrt{\rho}\partial_{t}u$. In fact, the time-weighted estimate is  crucial   in dropping the compatibility condition  on the initial data (see \cite{Hellsd,lijink,LsZhong,Paicu2} for example).
As a result,  (\ref{testiamt11}) allows us to derive the desired exponential decay-in-time rate (see (\ref{521tt007})):
$$e^{\gamma t}\|\sqrt{\rho}\partial_{t}u(t)\|_{L^{2}}^{2} +\int_{1}^{t}{e^{\gamma \tau}\|\Lambda^{\frac{1}{2}+\frac{n}{4}}\partial_{\tau}u(\tau)
\|_{L^{2}}^{2}\,d\tau} \leq C,\quad \forall\,t\geq1.$$
As a matter of fact, all these exponential
decay-in-time rates and the time-weighted estimate (\ref{testiamt11}) play an important role in obtaining the desired uniform-in-time bound of $\int_{0}^{t}{\|\nabla u(\tau)
\|_{L^{\infty}} \,d\tau}$ (see (\ref{xctoplk243}) for details). Next, by means of these a priori estimates, we can establish the time independent
estimates on the gradient of the density. This further allows us to derive the time-weighted estimate (see (\ref{t5529b01})):
\begin{eqnarray} \label{cvfdr67}
t^{2}\|\Lambda^{\frac{1}{2}+\frac{n}{4}}\partial_{t}u(t)
\|_{L^{2}}^{2} +\int_{0}^{t}{\tau^{2}\|\sqrt{\rho}\partial_{\tau\tau}u(\tau)\|_{L^{2}}^{2}\,d\tau} \leq C,\quad \forall\,t\geq0.
\end{eqnarray}
Note that, thanks to the weighted factor $t^{2}$, the constant $C$ in the above estimate is independent of the initial data of $\Lambda^{\frac{1}{2}+\frac{n}{4}}\partial_{t}u$.
With (\ref{cvfdr67}) in hand, we then can conclude the exponential decay-in-time rate (see (\ref{dsvv54t225})):
\begin{eqnarray}
e^{\gamma t}\|\Lambda^{\frac{1}{2}+\frac{n}{4}}\partial_{t}u(t)
\|_{L^{2}}^{2} +\int_{1}^{t}{e^{\gamma \tau}\|\sqrt{\rho}\partial_{\tau\tau}u(\tau)\|_{L^{2}}^{2}\,d\tau} \leq C ,\quad \forall\,t\geq1.\nonumber
\end{eqnarray}
Therefore, the higher regularity of the velocity and the pressure follow directly. The uniqueness is quite subtle as we only have the estimate  $\int_{0}^{t}{\tau\|\Lambda^{\frac{1}{2}+\frac{n}{4}}\partial_{\tau}u(\tau)
\|_{L^{2}}^{2}\,d\tau} \leq C$ rather than $\int_{0}^{t}{\|\Lambda^{\frac{1}{2}+\frac{n}{4}}\partial_{\tau}u(\tau)
\|_{L^{2}}^{2}\,d\tau} \leq C$. This means that the uniqueness can not be proved by the standard Gronwall's inequality,
instead  we use a new Gronwall type inequality in \cite{lijink}.
With all these a priori estimates obtained, we can finally establish the global existence and uniqueness as well as the exponential decay of global strong solution to the system \eqref{INDNSE} in Theorem \ref{Th1}.

As a byproduct, using the similar arguments of the proof for Theorem \ref{Th1}, we can also obtain the exponential decay of strong solutions to the two-dimensional Navier-Stokes equations with damping. We remark that without damping, only algebraic decay rate was obtained in \cite{ LsZhong}.

The rest of the paper is organized as follows. In Section 2 we carry out the proof of Theorem \ref{Th1}. In the appendix, we present
the byproduct on the exponential decay for the two-dimensional Navier-Stokes equations with damping and a sketch of the proof.

\bigskip

\section{The proof of Theorem \ref{Th1}}\setcounter{equation}{0}
This section is devoted to the proof of Theorem \ref{Th1}.
We shall prove Theorem \ref{Th1} in several steps. In the first step, we state the local existence and uniqueness of strong solutions. The main part of the proof will focus on establishing a priori estimates for strong solutions. In the second step, we make use of the estimate on the density to derive the exponential decay-in-time:  $e^{\gamma t}\|\sqrt{\rho}u(t)\|_{L^{2}}^{2}\leq C$ for some $\gamma>0$, which also allows us to further establish the same exponential
decay-in-time:  $e^{\gamma t}\|\Lambda^{\frac{1}{2}+\frac{n}{4}}u(t)\|_{L^{2}}^{2}\leq C$.
In the third step, with the aid of the exponential
decay estimates obtained above, we continue to derive the time-weighted estimates and the exponential
decay of $\|\sqrt{\rho}\partial_{t}u(t)\|_{L^{2}}^{2}$ as well as
some other quantities. With the above estimates at hand, the fourth step is devoted to obtaining the uniform-in-time bound of $\int_{0}^{t}{\|\nabla u(\tau)
\|_{L^{\infty}} \,d\tau}$ and thus establishing the estimate of the gradient of $\rho$.
In the fifth step, we establish the time-weighted estimates and the exponential
decay of $\|\Lambda^{\frac{1}{2}+\frac{n}{4}}\partial_{t}u(t)
\|_{L^{2}}^{2}$ and some other quantities. Finally, combining all the above estimates, we prove Theorem \ref{Th1}. Now we present the details step by step.

\subsection{Local well-posedness}\setcounter{equation}{0}

Inspired by the works \cite{CK2003,lijink}, one may construct the local existence and uniqueness of strong solutions.

\begin{lemma}[Local strong solution]\label{LT}
Under the conditions of Theorem \ref{Th1},  there exists a small time $T^{\ast}$ and a unique strong solution $(\rho,\,u)$ defined on the time period $[0,T^{\ast}]$ to the system (\ref{INDNSE}) with $\alpha=
\frac{1}{2}+\frac{n}{4}$ and $n\geq2$ such that,  for any $0<\tau<T^{\ast}$,
$$0\leq\rho\in L^{\infty}(0,T^{\ast}; L^{\frac{2n}{n+2}}(\mathbb{R}^{n})\cap L^{\infty}(\mathbb{R}^{n})),\quad \nabla\rho\in L^{\infty}(0,T^{\ast}; L^{\frac{4n}{n+6}}(\mathbb{R}^{n})\cap L^{2}(\mathbb{R}^{n})),$$
$$p \in L^{\infty}(0,T^{\ast}; {H}^{1}(\mathbb{R}^{n})),\quad u \in L^{\infty}(0,T^{\ast}; \dot{H}^{\frac{1}{2}+\frac{n}{4}}(\mathbb{R}^{n}))\cap L^{2}(0,T^{\ast}; \dot{H}^{1+\frac{n}{2}}(\mathbb{R}^{n})),$$
$$ \sqrt{\rho}\partial_{t}u\in L^{\infty}(\tau,T^{\ast}; L^{2}(\mathbb{R}^{n})),\quad
\partial_{t}u\in L^{2}(\tau,T^{\ast}; \dot{H}^{\frac{1}{2}+\frac{n}{4}}(\mathbb{R}^{n})).$$
\end{lemma}

\vskip .1in
\subsection{Exponential
decay of $\|\sqrt{\rho}u(t)\|_{L^{2}}^{2}$ and $\|\Lambda^{\frac{1}{2}+\frac{n}{4}}u(t)\|_{L^{2}}^{2}$}
We begin with the basic energy estimates.
\begin{lemma}\label{INDL221}
Under the assumptions of Theorem \ref{Th1}, the   solution $(\rho,u)$
of the system (\ref{INDNSE}) admits the following bound for any $t\geq0$,
\begin{eqnarray}&&
\|\rho(t)\|_{L^{\frac{2n}{n+2}}\cap L^{\infty}}\leq
\|\rho_{0}\|_{L^{\frac{2n}{n+2}}\cap L^{\infty}}, \label{addt030t002}\\
&&e^{\gamma t}\|\sqrt{\rho}u(t)\|_{L^{2}}^{2} + \int_{0}^{t}{e^{\gamma \tau}\|\Lambda^{\frac{1}{2}+\frac{n}{4}}u(\tau)\|_{L^{2}}^{2}\,d\tau}\leq  \|\sqrt{\rho_{0}}u_{0}\|_{L^{2}}^{2}.\label{addt030t001}
\end{eqnarray}
\end{lemma}

\begin{proof}
First, the non-negativeness of $\rho$ is a direct consequence of the maximum principle and $\rho_{0}\geq0$.
We multiply the equation $(\ref{INDNSE})_{1}$ by $|\rho|^{p-2}\rho$, integrate it over $\mathbb{R}^{n}$ and use $\nabla\cdot u=0$ to conclude
$$\frac{d}{dt}\|\rho(t)\|_{L^{p}}=0.$$
We then obtain
$\|\rho(t)\|_{L^{p}}\leq
\|\rho_{0}\|_{L^{p}}.$
Letting $p\rightarrow\infty$ yields
$\|\rho(t)\|_{L^{\infty}}\leq
\|\rho_{0}\|_{L^{\infty}}.$

In order to show (\ref{addt030t001}), we multiply equation $(\ref{INDNSE})_{2}$ by $u$, use the equation $(\ref{INDNSE})_{1}$ and integrate the resulting equation over $\mathbb{R}^{n}$ to show
\begin{eqnarray}\label{53ttyet003}
 \frac{1}{2}\frac{d}{dt} \|\sqrt{\rho}u(t)\|_{L^{2}}^{2} +\|\Lambda^{\frac{1}{2}+\frac{n}{4}}u\|_{L^{2}}^{2}=0.
\end{eqnarray}
Now it is easy to check
\begin{eqnarray}\label{key1}
\|\sqrt{\rho}u\|_{L^{2}}\leq \|\sqrt{\rho}\|_{L^{\frac{4n}{n+2}}}\|u\|_{L^{\frac{4n}{n-2}}} 
\leq C_{\star} \|{\rho}\|_{L^{\frac{2n}{n+2}}}^{\frac{1}{2}}
\|\Lambda^{\frac{1}{2}+\frac{n}{4}}u\|_{L^{2}}
\leq C_{\star}\|{\rho_{0}}\|_{L^{\frac{2n}{n+2}}}^{\frac{1}{2}}
\|\Lambda^{\frac{1}{2}+\frac{n}{4}}u\|_{L^{2}},
\end{eqnarray}
where $C_{\star}=C(n)$ is a  constant. Thus, we conclude from (\ref{53ttyet003}) that
\begin{eqnarray*}\label{53ttyet004}
 \frac{d}{dt} \|\sqrt{\rho}u(t)\|_{L^{2}}^{2} +\gamma\|\sqrt{\rho}u(t)\|_{L^{2}}^{2} +\|\Lambda^{\frac{1}{2}+\frac{n}{4}}u\|_{L^{2}}^{2}=0,
\end{eqnarray*}
where 
$$\gamma=\frac{1}{C_{\star}^{2}\|{\rho_{0}}\|_{L^{\frac{2n}{n+2}}}}.$$
An application of the Gronwall inequality yields \eqref{addt030t001}.
This completes the proof of Lemma \ref{INDL221}.
\end{proof}

Based on the estimate (\ref{addt030t001}), we now derive the same exponential decay estimate for $\|\Lambda^{\frac{1}{2}+\frac{n}{4}}u(t)\|_{L^{2}}^{2}$.
\begin{lemma}\label{INDL222}
Under the assumptions of Theorem \ref{Th1}, the   solution $(\rho,u)$
of the system (\ref{INDNSE}) admits the following bound for any $t\geq0$,
\begin{eqnarray} \label{addt030t003}
e^{\gamma t}\|\Lambda^{\frac{1}{2}+\frac{n}{4}}u(t)\|_{L^{2}}^{2} +\int_{0}^{t}{e^{\gamma \tau}(\|\Lambda^{1+\frac{n}{2}}u(\tau)\|_{L^{2}}^{2}+\|\sqrt{\rho}
\partial_{\tau}u(\tau)\|_{L^{2}}^{2})\,d\tau}\leq \widetilde{C_{1}},
\end{eqnarray}
where $\widetilde{C_{1}}$ depends only on $\|\rho_{0}\|_{L^{\frac{2n}{n+2}}}$, $\|\rho_{0}\|_{L^{\infty}}$, $\|\sqrt{\rho_{0}}u_{0}\|_{L^{2}}$ and $\|\Lambda^{\frac{1}{2}+\frac{n}{4}}u_{0}\|_{L^{2}}$.
\end{lemma}

\begin{proof}
First, multiplying the equation $(\ref{INDNSE})_{2}$ by $\partial_{t}u$, using $\nabla\cdot u=0$ and integrating by parts, we obtain
\begin{eqnarray}
 \frac{1}{2}\frac{d}{dt} \|\Lambda^{\frac{1}{2}+\frac{n}{4}}u(t)\|_{L^{2}}^{2} +\|\sqrt{\rho}
\partial_{t}u\|_{L^{2}}^{2}=-\int_{\mathbb{R}^{n}}\rho u\cdot\nabla u\cdot \partial_{t}u\,dx.\nonumber
\end{eqnarray}
With the aid of the Gagliardo-Nirenberg inequality, one gets
\begin{align} \label{cvrfyuq6}
-\int_{\mathbb{R}^{n}}\rho u\cdot\nabla u\cdot \partial_{t}u\,dx&\leq \|u \cdot \nabla u\|_{L^{2}}\|\sqrt{\rho}
\|_{L^{\infty}}\|\sqrt{\rho}
\partial_{t}u\|_{L^{2}}\nonumber\\
&\leq C\|{\rho_{0}}
\|_{L^{\infty}}^{\frac{1}{2}}\|u\|_{L^{\frac{4n}{n-2}}}\|\nabla u\|_{L^{\frac{4n}{n+2}}}\|\sqrt{\rho}
\partial_{t}u\|_{L^{2}}
\nonumber\\
&\leq C \|\Lambda^{\frac{1}{2}+\frac{n}{4}}u\|_{L^{2}}^{2}\|\sqrt{\rho}
\partial_{t}u\|_{L^{2}}\nonumber\\
&\leq \frac{1}{2}\|\sqrt{\rho}
\partial_{t}u\|_{L^{2}}^{2}+C \|\Lambda^{\frac{1}{2}+\frac{n}{4}}u\|_{L^{2}}^{2}
\|\Lambda^{\frac{1}{2}+\frac{n}{4}}u\|_{L^{2}}^{2}.
\end{align}
We therefore conclude that
\begin{eqnarray}
 \frac{d}{dt} \|\Lambda^{\frac{1}{2}+\frac{n}{4}}u(t)\|_{L^{2}}^{2} +\|\sqrt{\rho}
\partial_{t}u\|_{L^{2}}^{2}\leq C \|\Lambda^{\frac{1}{2}+\frac{n}{4}}u\|_{L^{2}}^{2}
\|\Lambda^{\frac{1}{2}+\frac{n}{4}}u\|_{L^{2}}^{2}.\nonumber
\end{eqnarray}
This implies
\begin{align*}
 \frac{d}{dt} (e^{\gamma t}\|\Lambda^{\frac{1}{2}+\frac{n}{4}}u(t)\|_{L^{2}}^{2}) +e^{\gamma t}\|\sqrt{\rho}
\partial_{t}u\|_{L^{2}}^{2}&\leq  \gamma e^{\gamma t}\|\Lambda^{\frac{1}{2}+\frac{n}{4}}u(t)\|_{L^{2}}^{2}\nonumber\\& \quad+C e^{\gamma t} \|\Lambda^{\frac{1}{2}+\frac{n}{4}}u\|_{L^{2}}^{2}
\|\Lambda^{\frac{1}{2}+\frac{n}{4}}u\|_{L^{2}}^{2}.
\end{align*}
Integrating in time and using (\ref{addt030t001}) yield
\begin{align} &
e^{\gamma t}\|\Lambda^{\frac{1}{2}+\frac{n}{4}}u(t)\|_{L^{2}}^{2} +\int_{0}^{t}e^{\gamma \tau}\|\sqrt{\rho}
\partial_{\tau}u(\tau)\|_{L^{2}}^{2}\,d\tau \nonumber\\&\leq \|\Lambda^{\frac{1}{2}+\frac{n}{4}}u_{0}\|_{L^{2}}^{2}+\gamma \int_{0}^{t}e^{\gamma \tau}\|\Lambda^{\frac{1}{2}+\frac{n}{4}}u(\tau)\|_{L^{2}}^{2}\,d\tau
\nonumber\\   & \quad + C\int_{0}^{t} e^{\gamma \tau} \|\Lambda^{\frac{1}{2}+\frac{n}{4}}u(\tau)\|_{L^{2}}^{2}
\|\Lambda^{\frac{1}{2}+\frac{n}{4}}u(\tau)\|_{L^{2}}^{2}\,d\tau
\nonumber\\ &\leq \widetilde{C}
+C\int_{0}^{t} e^{\gamma \tau} \|\Lambda^{\frac{1}{2}+\frac{n}{4}}u(\tau)\|_{L^{2}}^{2}
\|\Lambda^{\frac{1}{2}+\frac{n}{4}}u(\tau)\|_{L^{2}}^{2}\,d\tau.\nonumber
\end{align}
We thus get
\begin{align}\label{qfgtds3c}
&e^{\gamma t}\|\Lambda^{\frac{1}{2}+\frac{n}{4}}u(t)\|_{L^{2}}^{2} +\int_{0}^{t}e^{\gamma \tau}\|\sqrt{\rho}
\partial_{\tau}u(\tau)\|_{L^{2}}^{2}\,d\tau \nonumber\\
&\leq \widetilde{C}
+C\int_{0}^{t} e^{\gamma \tau} \|\Lambda^{\frac{1}{2}+\frac{n}{4}}u(\tau)\|_{L^{2}}^{2}
\|\Lambda^{\frac{1}{2}+\frac{n}{4}}u(\tau)\|_{L^{2}}^{2}\,d\tau.
\end{align}
By virtue of the Gronwall inequality and (\ref{addt030t001}), one has
\begin{eqnarray}\label{qfgtds3c11} e^{\gamma t}\|\Lambda^{\frac{1}{2}+\frac{n}{4}}u(t)\|_{L^{2}}^{2} \leq \widetilde{C_{1}}\exp\left[\int_{0}^{t}{
\|\Lambda^{\frac{1}{2}+\frac{n}{4}}u(\tau)\|_{L^{2}}^{2} \,d\tau}\right]\leq \widetilde{C_{1}},\end{eqnarray}
which along with (\ref{qfgtds3c}) also implies
\begin{eqnarray}\label{qfgtds3c22} \int_{0}^{t}e^{\gamma \tau}\|\sqrt{\rho}
\partial_{\tau}u(\tau)\|_{L^{2}}^{2}\,d\tau\leq \widetilde{C_{1}}.\end{eqnarray}
Now let us recall the Stokes equations
\begin{equation}\label{521tt005}
\left\{\begin{array}{l}
(-\Delta)^{\frac{1}{2}+\frac{n}{4}}u+\nabla p =-\rho\partial_tu
-\rho u\cdot\nabla u,
             \vspace{2mm}\\
\nabla\cdot u=0.
\end{array}\right.
\end{equation}
Applying the regularity properties of the Stokes system (\ref{521tt005}), it follows that
\begin{align}\label{csroplk}
\|\Lambda^{1+\frac{n}{2}}u\|_{L^{2}}&\leq C\|\rho\partial_tu\|_{L^{2}}
+C\|\rho u\cdot\nabla u\|_{L^{2}}\nonumber\\
&\leq C\|\sqrt{\rho}
\|_{L^{\infty}}\|\sqrt{\rho}\partial_tu\|_{L^{2}}
+C\|{\rho}
\|_{L^{\infty}}\| u\cdot\nabla u\|_{L^{2}}
\nonumber\\
&\leq C\|\sqrt{\rho}\partial_tu\|_{L^{2}}
+C\|u\|_{L^{\frac{4n}{n-2}}}\|\nabla u\|_{L^{\frac{4n}{n+2}}}\nonumber\\
&\leq C\|\sqrt{\rho}\partial_tu\|_{L^{2}}
+C\|\Lambda^{\frac{1}{2}+\frac{n}{4}}u\|_{L^{2}}^{2}.
\end{align}
This allows us to show
\begin{align} \label{521tt006}
 \int_{0}^{t}{e^{\gamma \tau}\|\Lambda^{1+\frac{n}{2}}u(\tau)\|_{L^{2}}^{2}\,d\tau}&\leq \int_{0}^{t}{e^{\gamma \tau}\|\sqrt{\rho}\partial_\tau u(\tau)\|_{L^{2}}^{2}\,d\tau}+\int_{0}^{t}{e^{\gamma \tau}\|\Lambda^{\frac{1}{2}+\frac{n}{4}}u(\tau)\|_{L^{2}}^{4}\,d\tau}
 \nonumber\\&\leq  \widetilde{C_{1}},
\end{align}
where we have used (\ref{addt030t001}), (\ref{qfgtds3c11}) and (\ref{qfgtds3c22}).
We thus complete the proof of the lemma by combining (\ref{qfgtds3c11}), (\ref{qfgtds3c22}) and (\ref{521tt006}).
\end{proof}

\subsection{Time-weighted estimates and exponential
decay of $\|\sqrt{\rho}\partial_{t}u(t)\|_{L^{2}}^{2}$ and other quantities}
The following lemma is crucial to derive the higher order estimates of the solutions.
\begin{lemma}\label{INDL223}
Under the assumptions of Theorem \ref{Th1}, the   solution $(\rho,u)$
of the system (\ref{INDNSE}) admits the following bound for any $t\geq0$,
\begin{eqnarray} \label{addok5251}
t\|\Lambda^{1+\frac{n}{2}}u(t)\|_{L^{2}}^{2}+t\|p(t)\|_{H^{1}}^{2}+t\|\sqrt{\rho}
\partial_{t}u(t)\|_{L^{2}}^{2} +\int_{0}^{t}{\tau\|\Lambda^{\frac{1}{2}+\frac{n}{4}}\partial_{\tau}u(\tau)
\|_{L^{2}}^{2}\,d\tau} \leq \widetilde{C_{1}}.
\end{eqnarray}
Moreover, for any $t\geq1$, the following estimates hold true
\begin{gather}
e^{\gamma t}\|\sqrt{\rho}\partial_{t}u(t)\|_{L^{2}}^{2} +\int_{1}^{t}{e^{\gamma \tau}\|\Lambda^{\frac{1}{2}+\frac{n}{4}}\partial_{\tau}u(\tau)
\|_{L^{2}}^{2}\,d\tau} \leq \widetilde{C_{1}}, \label{521tt007}\\
e^{\gamma t}
\|\Lambda^{1+\frac{n}{2}}u(t)\|_{L^{2}}^{2}+\int_{1}^{t}{e^{\gamma \tau}\|\Lambda^{1+\frac{n}{2}} u(\tau)
\|_{L^{2}}^{2}\,d\tau} \leq \widetilde{C_{1}}, \label{xfrmu8}\\
e^{\gamma t}
\|p(t)\|_{H^{1}}^{2} \leq \widetilde{C_{1}}, \label{xfxervg}
\end{gather}
where $\widetilde{C_{1}}$ depends only on $\|\rho_{0}\|_{L^{\frac{2n}{n+2}}}$, $\|\rho_{0}\|_{L^{\infty}}$, $\|\sqrt{\rho_{0}}u_{0}\|_{L^{2}}$ and $\|\Lambda^{\frac{1}{2}+\frac{n}{4}}u_{0}\|_{L^{2}}$.
\end{lemma}

\begin{proof}
First, applying the time derivative $\partial_{t}$ to the equation $(\ref{INDNSE})_{2}$ gives
\begin{equation}\label{521tt008}
\rho\partial_{tt}u+\rho u\cdot\nabla \partial_{t}u
+(-\Delta)^{\frac{1}{2}+\frac{n}{4}}\partial_{t}u+\nabla \partial_{t}p =-\partial_{t}\rho\partial_tu
-\partial_{t}(\rho u)\cdot\nabla u.
\end{equation}
Multiplying (\ref{521tt008}) by $\partial_{t}u$ and using the equation $(\ref{INDNSE})_{1}$, we derive that
\begin{align}& \label{5yion6de}
 \frac{1}{2}\frac{d}{dt} \|\sqrt{\rho}\partial_{t}u(t)\|_{L^{2}}^{2} +\|\Lambda^{\frac{1}{2}+\frac{n}{4}}\partial_{t}u\|_{L^{2}}^{2}\nonumber\\&=
-\int_{\mathbb{R}^{n}}\partial_{t}\rho\partial_{t}u\cdot \partial_{t}u\,dx-\int_{\mathbb{R}^{n}}\partial_{t}(\rho u)\cdot\nabla u\cdot \partial_{t}u\,dx\nonumber\\
 &= -2\int_{\mathbb{R}^{n}}\rho u\cdot\nabla \partial_{t}u\cdot \partial_{t}u\,dx-\int_{\mathbb{R}^{n}}\rho \partial_{t}u\cdot\nabla u\cdot \partial_{t}u\,dx
 -\int_{\mathbb{R}^{n}}\rho u\cdot\nabla(u\cdot\nabla u\cdot \partial_{t}u)\,dx
 \nonumber\\
 &:= N_{1}+N_{2}+N_{3}.
\end{align}
By means of the embedding inequalities, one shows
\begin{align}
N_{1}&\leq  C\|\sqrt{\rho}\|_{L^{\infty}}\|\sqrt{\rho}\partial_tu\|_{L^{2}}
\|\nabla\partial_tu\|_{L^{\frac{4n}{n+2}}}\|u\|_{L^{\frac{4n}{n-2}}}\nonumber\\
&\leq  C\|\sqrt{\rho}\partial_tu\|_{L^{2}}
\|\Lambda^{\frac{1}{2}+\frac{n}{4}}\partial_tu\|_{L^{2}}\|\Lambda^{\frac{1}{2}
+\frac{n}{4}}
u\|_{L^{2}}
\nonumber\\
&\leq \frac{1}{8}\|\Lambda^{\frac{1}{2}+\frac{n}{4}}\partial_{t}u\|_{L^{2}}^{2}+
C\|\Lambda^{\frac{1}{2}+\frac{n}{4}}u\|_{L^{2}}^{2}\|\sqrt{\rho}\partial_tu\|_{L^{2}}^{2}
\nonumber
\end{align}
and
\begin{align}
N_{2}&\leq  C\|\sqrt{\rho}\|_{L^{\infty}}\|\sqrt{\rho}\partial_tu\|_{L^{2}}
\|\nabla u\|_{L^{\frac{4n}{n+2}}}\|\partial_tu\|_{L^{\frac{4n}{n-2}}}\nonumber\\
&\leq  C\|\sqrt{\rho}\partial_tu\|_{L^{2}}
\|\Lambda^{\frac{1}{2}+\frac{n}{4}}u\|_{L^{2}}\|\Lambda^{\frac{1}{2}
+\frac{n}{4}}
\partial_tu\|_{L^{2}}
\nonumber\\
&\leq \frac{1}{8}\|\Lambda^{\frac{1}{2}+\frac{n}{4}}\partial_{t}u\|_{L^{2}}^{2}+
C\|\Lambda^{\frac{1}{2}+\frac{n}{4}}u\|_{L^{2}}^{2}\|\sqrt{\rho}
\partial_tu\|_{L^{2}}^{2}
.\nonumber
\end{align}
For the term $N_{3}$, it can be bounded by
\begin{align}
N_{3}&\leq \left|
\int_{\mathbb{R}^{n}}\rho u\cdot\nabla u\cdot\nabla u\cdot \partial_{t}u\,dx\right|+\left|
\int_{\mathbb{R}^{n}}\rho u\cdot u\cdot\nabla^{2} u\cdot \partial_{t}u\,dx\right|
\nonumber\\& \quad+\left|
\int_{\mathbb{R}^{n}}\rho u\cdot u\cdot\nabla u\cdot \nabla\partial_{t}u\,dx\right|\nonumber\\&\leq C
\|{\rho}\|_{L^{\infty}}\|u\|_{L^{\frac{4n}{n-2}}}
\|\nabla u\|_{L^{\frac{4n}{n+2}}}^{2}\|\partial_tu\|_{L^{\frac{4n}{n-2}}}+
C\|\sqrt{\rho}\|_{L^{\infty}}\|\sqrt{\rho}\partial_tu\|_{L^{2}}
\|\nabla^{2} u\|_{L^{n}}\|u\|_{L^{\frac{4n}{n-2}}}^{2}\nonumber\\& \quad+C
\|{\rho}\|_{L^{\infty}}\|u\|_{L^{\frac{4n}{n-2}}}^{2}
\|\nabla u\|_{L^{\frac{4n}{n+2}}}\|\nabla\partial_tu\|_{L^{\frac{4n}{n+2}}}
\nonumber\\&\leq C\|\Lambda^{\frac{1}{2}+\frac{n}{4}}u\|_{L^{2}}
\|\Lambda^{\frac{1}{2}+\frac{n}{4}}u\|_{L^{2}}^{2}
\|\Lambda^{\frac{1}{2}+\frac{n}{4}}\partial_{t}u\|_{L^{2}}+
C\|\sqrt{\rho}\partial_tu\|_{L^{2}}
\|\Lambda^{1+\frac{n}{2}}u\|_{L^{2}}\|\Lambda^{\frac{1}{2}+\frac{n}{4}}u\|_{L^{2}}^{2}
\nonumber\\& \quad+C
\|\Lambda^{\frac{1}{2}+\frac{n}{4}}u\|_{L^{2}}^{2}
\|\Lambda^{\frac{1}{2}+\frac{n}{4}}u\|_{L^{2}}
\|\Lambda^{\frac{1}{2}+\frac{n}{4}}\partial_{t}u\|_{L^{2}}
\nonumber\\&\leq C\|\Lambda^{\frac{1}{2}+\frac{n}{4}}u\|_{L^{2}}
\|\Lambda^{\frac{1}{2}+\frac{n}{4}}u\|_{L^{2}}^{2}
\|\Lambda^{\frac{1}{2}+\frac{n}{4}}\partial_{t}u\|_{L^{2}}+
C\|\sqrt{\rho}\partial_tu\|_{L^{2}}^{2}
\|\Lambda^{\frac{1}{2}+\frac{n}{4}}u\|_{L^{2}}^{2}
\nonumber\\& \quad +C\|\sqrt{\rho}\partial_tu\|_{L^{2}}
\|\Lambda^{\frac{1}{2}+\frac{n}{4}}u\|_{L^{2}}^{4}+C
\|\Lambda^{\frac{1}{2}+\frac{n}{4}}u\|_{L^{2}}^{2}
\|\Lambda^{\frac{1}{2}+\frac{n}{4}}u\|_{L^{2}}
\|\Lambda^{\frac{1}{2}+\frac{n}{4}}\partial_{t}u\|_{L^{2}}
\nonumber\\
&\leq \frac{1}{8}\|\Lambda^{\frac{1}{2}+\frac{n}{4}}\partial_{t}u\|_{L^{2}}^{2}+
C\|\Lambda^{\frac{1}{2}+\frac{n}{4}}u\|_{L^{2}}^{6}
+C\|\Lambda^{\frac{1}{2}+\frac{n}{4}}u\|_{L^{2}}^{2}\|\sqrt{\rho}
\partial_tu\|_{L^{2}}^{2},
\nonumber
\end{align}
where we have used the following fact due to (\ref{csroplk})
\begin{eqnarray}\label{xfrdf217}
\|\Lambda^{1+\frac{n}{2}}u\|_{L^{2}}\leq C \|\sqrt{\rho}\partial_tu\|_{L^{2}}
+C\|\Lambda^{\frac{1}{2}+\frac{n}{4}}u\|_{L^{2}}^{2}.
\end{eqnarray}
Substituting the above estimates into (\ref{5yion6de}) yields
\begin{eqnarray}\label{vcgy7uode}
\frac{d}{dt} \|\sqrt{\rho}\partial_{t}u(t)\|_{L^{2}}^{2} +\|\Lambda^{\frac{1}{2}+\frac{n}{4}}\partial_{t}u\|_{L^{2}}^{2}\leq C\|\Lambda^{\frac{1}{2}+\frac{n}{4}}u\|_{L^{2}}^{2}\|\sqrt{\rho}
\partial_tu\|_{L^{2}}^{2}+C\|\Lambda^{\frac{1}{2}+\frac{n}{4}}u\|_{L^{2}}^{6},
\end{eqnarray}
which implies
\begin{align*}
\frac{d}{dt}(t\|\sqrt{\rho}\partial_{t}u(t)\|_{L^{2}}^{2}) +t\|\Lambda^{\frac{1}{2}+\frac{n}{4}}\partial_{t}u(t)\|_{L^{2}}^{2}
\leq& C\|\Lambda^{\frac{1}{2}+\frac{n}{4}}u\|_{L^{2}}^{2}(t\|\sqrt{\rho}
\partial_tu\|_{L^{2}}^{2})\\
& +\|\sqrt{\rho}\partial_{t}u(t)\|_{L^{2}}^{2}+Ct\|\Lambda^{\frac{1}{2}+\frac{n}{4}}u\|_{L^{2}}^{6}.
\end{align*}
From (\ref{addt030t001}) and (\ref{addt030t003}),  and by the Gronwall inequality, one has
\begin{eqnarray}\label{vcxcdft67e}
t\|\sqrt{\rho}\partial_{t}u(t)\|_{L^{2}}^{2} +\int_{0}^{t}{\tau\|\Lambda^{\frac{1}{2}+\frac{n}{4}}\partial_{\tau}u(\tau)
\|_{L^{2}}^{2}\,d\tau} \leq \widetilde{C_{1}}.\end{eqnarray}
Moreover, we deduce from (\ref{vcgy7uode}) that
\begin{align*}
\frac{d}{dt}(e^{\gamma t}\|\sqrt{\rho}\partial_{t}u(t)\|_{L^{2}}^{2}) +e^{\gamma t}\|\Lambda^{\frac{1}{2}+\frac{n}{4}}\partial_{t}u(t)\|_{L^{2}}^{2}\leq & C\|\Lambda^{\frac{1}{2}+\frac{n}{4}}u\|_{L^{2}}^{2}(e^{\gamma t}\|\sqrt{\rho}
\partial_tu\|_{L^{2}}^{2})\nonumber\\ & +\gamma e^{\gamma t}\|\sqrt{\rho}\partial_{t}u(t)\|_{L^{2}}^{2}+Ce^{\gamma t}\|\Lambda^{\frac{1}{2}+\frac{n}{4}}u\|_{L^{2}}^{6}.
\end{align*}
Integrating it in time and making use of (\ref{addt030t003}) as well as (\ref{vcxcdft67e}) lead to
\begin{align}&
e^{\gamma t}\|\sqrt{\rho}\partial_{t}u(t)\|_{L^{2}}^{2} +\int_{1}^{t}e^{\gamma \tau}\|\Lambda^{\frac{1}{2}+\frac{n}{4}}\partial_{\tau}u(\tau)\|_{L^{2}}^{2}\,d\tau
\nonumber\\&\leq  \widetilde{C_{1}}+ C\int_{1}^{t}\|\Lambda^{\frac{1}{2}+\frac{n}{4}}u(\tau)\|_{L^{2}}^{2}(e^{\gamma \tau}\|\sqrt{\rho}
\partial_\tau u(\tau)\|_{L^{2}}^{2})\,d\tau+\gamma \int_{1}^{t}e^{\gamma \tau}\|\sqrt{\rho}\partial_{\tau}u(\tau)\|_{L^{2}}^{2}\,d\tau\nonumber\\& \quad +C\int_{1}^{t}(e^{\gamma \tau}\|\Lambda^{\frac{1}{2}+\frac{n}{4}}u(\tau)\|_{L^{2}}^{2})
\|\Lambda^{\frac{1}{2}+\frac{n}{4}}u(\tau)\|_{L^{2}}^{4}\,d\tau
\nonumber\\&\leq  \widetilde{C_{1}}+ C\int_{1}^{t}\|\Lambda^{\frac{1}{2}+\frac{n}{4}}u(\tau)\|_{L^{2}}^{2}(e^{\gamma \tau}\|\sqrt{\rho}
\partial_\tau u(\tau)\|_{L^{2}}^{2})\,d\tau.
\nonumber
\end{align}
By the same argument adopted in dealing with (\ref{qfgtds3c11}) and (\ref{qfgtds3c22}), we thus deduce
$$e^{\gamma t}\|\sqrt{\rho}\partial_{t}u(t)\|_{L^{2}}^{2} +\int_{1}^{t}{e^{\gamma \tau}\|\Lambda^{\frac{1}{2}+\frac{n}{4}}\partial_{\tau}u(\tau)
\|_{L^{2}}^{2}\,d\tau} \leq \widetilde{C_{1}}.$$
By means of (\ref{xfrdf217}), (\ref{521tt007}) and (\ref{addt030t003}), we have
\begin{align} &
e^{\gamma t}
\|\Lambda^{1+\frac{n}{2}}u\|_{L^{2}}^{2}+\int_{1}^{t}{e^{\gamma \tau}\|\Lambda^{1+\frac{n}{2}} u(\tau)
\|_{L^{2}}^{2}\,d\tau} \nonumber\\&\leq
Ce^{\gamma t}(\|\sqrt{\rho}\partial_tu\|_{L^{2}}^{2}
+\|\Lambda^{\frac{1}{2}+\frac{n}{4}}u\|_{L^{2}}^{4})+\int_{1}^{t}{e^{\gamma \tau}(\|\sqrt{\rho}\partial_\tau u(\tau)\|_{L^{2}}^{2}
+\|\Lambda^{\frac{1}{2}+\frac{n}{4}}u(\tau)\|_{L^{2}}^{4})\,d\tau}\nonumber\\&=
Ce^{\gamma t}\|\sqrt{\rho}\partial_tu\|_{L^{2}}^{2}
+Ce^{-\gamma t}(e^{\gamma t}\|\Lambda^{\frac{1}{2}+\frac{n}{4}}u\|_{L^{2}}^{2})^{2}\nonumber\\& \quad
+\int_{1}^{t}{\left(e^{\gamma \tau}\|\sqrt{\rho}\partial_\tau u(\tau)\|_{L^{2}}^{2}
+e^{-\gamma \tau}(e^{\gamma \tau}\|\Lambda^{\frac{1}{2}+\frac{n}{4}}u(\tau)\|_{L^{2}}^{2})
(e^{\gamma \tau}\|\Lambda^{\frac{1}{2}+\frac{n}{4}}u(\tau)\|_{L^{2}}^{2})\right)\,d\tau}
\nonumber\\&\leq
\widetilde{C_{1}}+\widetilde{C_{1}}\int_{1}^{t}{e^{\gamma \tau}\|\Lambda^{\frac{1}{2}+\frac{n}{4}}u(\tau)\|_{L^{2}}^{2}\,d\tau}\nonumber\\&\leq
\widetilde{C_{1}}.\nonumber
\end{align}
We thus obtain (\ref{xfrmu8}).
It follows from the Stokes system (\ref{521tt005}) that
\begin{align*}
\|\nabla p\|_{L^{2}}\leq C\|\rho\partial_tu\|_{L^{2}}
+C\|\rho u\cdot\nabla u\|_{L^{2}}
\leq C\|\sqrt{\rho}\partial_tu\|_{L^{2}}
+C \|\Lambda^{\frac{1}{2}+\frac{n}{4}}u\|_{L^{2}}^{2},
\end{align*}
where we have used the estimates in (\ref{cvrfyuq6}).
Similarly, we obtain
\begin{align}
\|p\|_{L^{2}}&\leq  C\|\Lambda^{-1}(\rho\partial_tu\|_{L^{2}})\|_{L^{2}}
+C\|\Lambda^{-1}(\rho u\cdot\nabla u)\|_{L^{2}}\nonumber\\
&\leq C\|\rho\partial_tu\|_{L^{\frac{2n}{n+2}}}
+C\|\rho u\cdot\nabla u\|_{L^{\frac{2n}{n+2}}}
\nonumber\\
&\leq  C\|\sqrt{\rho}\|_{L^{n}}\|\sqrt{\rho}\partial_tu\|_{L^{2}}
+C\|\rho \|_{L^{n}}\|u\cdot\nabla u\|_{L^{2}}
\nonumber\\
&\leq  C\|\sqrt{\rho_{0}}\|_{L^{n}}\|\sqrt{\rho}\partial_tu\|_{L^{2}}
+C\|\rho_{0}\|_{L^{n}}\|u\cdot\nabla u\|_{L^{2}}
\nonumber\\
&\leq C\|\sqrt{\rho}\partial_tu\|_{L^{2}}
+C \|\Lambda^{\frac{1}{2}+\frac{n}{4}}u\|_{L^{2}}^{2}.\nonumber
\end{align}
As before, we therefore obtain for all $t\geq1$,
\begin{eqnarray}
e^{\gamma t}
\|p\|_{H^{1}}^{2}\leq Ce^{\gamma t}\|\sqrt{\rho}\partial_tu\|_{L^{2}}^{2}
+C e^{\gamma t}\|\Lambda^{\frac{1}{2}+\frac{n}{4}}u\|_{L^{2}}^{4} \leq \widetilde{C_{1}}.\nonumber
\end{eqnarray}
This completes the proof of Lemma \ref{INDL223}.
\end{proof}

\subsection{Uniform in time bound of $\int_{0}^{t}{\|\nabla u(\tau)\|_{L^{\infty}} \,d\tau}$ and gradient of $\rho$}

The following estimates will be used to show the uniqueness of  solutions and the exponential decay of other quantities.
\begin{lemma}\label{INDL224}
Under the assumptions of Theorem \ref{Th1}, the   solution $(\rho,u)$
of the system (\ref{INDNSE}) admits the following bounds for any $t\geq0$,
\begin{gather}
\int_{0}^{t}{\|\nabla u(\tau)
\|_{L^{\infty}} \,d\tau} \leq \widetilde{C_{1}}, \label{xctoplk243}\\
\|\nabla\rho(t)\|_{L^{\frac{4n}{n+6}}} \leq \widetilde{C_{1}}\|\nabla\rho_{0}\|_{L^{\frac{4n}{n+6}}}, \label{521tt009}\\
\|\nabla\rho(t)\|_{L^{2}}  \leq \widetilde{C_{1}}\|\nabla\rho_{0}\|_{L^{2}}, \label{adds521tt009}
\end{gather}
where $\widetilde{C_{1}}$ depends only on $\|\rho_{0}\|_{L^{\frac{2n}{n+2}}}$, $\|\rho_{0}\|_{L^{\infty}}$, $\|\sqrt{\rho_{0}}u_{0}\|_{L^{2}}$ and $\|\Lambda^{\frac{1}{2}+\frac{n}{4}}u_{0}\|_{L^{2}}$.
\end{lemma}
\begin{rem}\rm
We remark that the bound (\ref{521tt009}) will be used to show the uniqueness, while the bound (\ref{adds521tt009}) will be used to derive the exponential decay of $\|\Lambda^{1+\frac{n}{2}}\partial_{t}u(t)\|_{L^{2}}^{2}$ and other quantities.
\end{rem}

\begin{proof}
First, it is easy to check that for any $2<p<\frac{4n}{n-2}$,
\begin{align} \label{keyusef1}
\|\rho\partial_tu\|_{L^{p}}&\leq C\|\rho\partial_tu\|_{L^{2}}^{1-\frac{2n(p-2)}
{(n+2)p}}
\|\rho\partial_tu\|_{L^{\frac{4n}{n-2}}}^{\frac{2n(p-2)}{(n+2)p}}\nonumber\\
&\leq C\|\sqrt{\rho}\|_{L^{\infty}}^{1-\frac{2n(p-2)}{(n+2)p}}\|\sqrt{\rho}
\partial_tu\|_{L^{2}}^{1-\frac{2n(p-2)}{(n+2)p}}
\|{\rho}\|_{L^{\infty}}^{\frac{2n(p-2)}{(n+2)p}}\|\partial_tu
\|_{L^{\frac{4n}{n-2}}}^{\frac{2n(p-2)}{(n+2)p}}
\nonumber\\
&\leq C
\|\sqrt{\rho}\partial_tu\|_{L^{2}}^{1-\frac{2n(p-2)}{(n+2)p}}
\|\Lambda^{\frac{1}{2}+\frac{n}{4}}\partial_tu\|_{L^{2}}^{\frac{2n(p-2)}{(n+2)p}}
\nonumber\\
&\leq C
\|\sqrt{\rho}\partial_tu\|_{L^{2}} +C
\|\Lambda^{\frac{1}{2}+\frac{n}{4}}\partial_tu\|_{L^{2}},
\end{align}
\begin{align}\label{keyusef2}
\|\rho u\cdot\nabla u\|_{L^{p}}&\leq C\|\rho u\cdot\nabla u\|_{L^{2}}^{1-\frac{2n(p-2)}{(n+2)p}}
\|\rho u\cdot\nabla u\|_{L^{\frac{4n}{n-2}}}^{\frac{2n(p-2)}{(n+2)p}}
\nonumber\\
&\leq C
\| u\cdot\nabla u\|_{L^{2}}^{1-\frac{2n(p-2)}{(n+2)p}}
\| u\cdot\nabla u\|_{L^{\frac{4n}{n-2}}}^{\frac{2n(p-2)}{(n+2)p}}
\nonumber\\
&\leq C
\|\Lambda^{\frac{1}{2}+\frac{n}{4}}u\|_{L^{2}}^{2\left(1-\frac{2n(p-2)}{(n+2)p}\right)}
\left(\|\Lambda^{\frac{1}{2}+\frac{n}{4}}u\|_{L^{2}}
\|\Lambda^{1+\frac{n}{2}}u\|_{L^{2}} \right)^{\frac{2n(p-2)}{(n+2)p}}
\nonumber\\
&\leq C
\|\Lambda^{\frac{1}{2}+\frac{n}{4}}u\|_{L^{2}}^{2}
+C\|\Lambda^{1+\frac{n}{2}}u\|_{L^{2}}^{2},
\end{align}
where we have used the following fact
\begin{align}\label{dfcdsadr7}
\| u\cdot\nabla u\|_{L^{\frac{4n}{n-2}}}&\leq  C\| u\|_{L^{\infty}}\|\nabla u\|_{L^{\frac{4n}{n-2}}}\nonumber\\&\leq  C(\|\Lambda^{\frac{1}{2}+\frac{n}{4}}u\|_{L^{2}}^{\frac{4}{n+2}}
\|\Lambda^{1+\frac{n}{2}}u\|_{L^{2}}^{1-\frac{4}{n+2}})(\|\Lambda^{\frac{1}{2}+\frac{n}{4}}u\|_{L^{2}}^{1-\frac{4}{n+2}}
\|\Lambda^{1+\frac{n}{2}}u\|_{L^{2}}^{\frac{4}{n+2}})
\nonumber\\&\leq
\|\Lambda^{\frac{1}{2}+\frac{n}{4}}u\|_{L^{2}}
\|\Lambda^{1+\frac{n}{2}}u\|_{L^{2}}.
\end{align}
Combining  the estimates (\ref{addt030t001}), (\ref{addt030t003}), (\ref{521tt007}) and (\ref{xfrmu8})  allows us to show that, for any $2<p<\frac{4n}{n-2}$ and for any $t\geq0$,
\begin{align} \label{521tt010}&
\int_{0}^{t}{(\|\rho\partial_tu(\tau)\|_{L^{p}}+\|\rho u\cdot\nabla u(\tau)\|_{L^{p}}) \,d\tau}\nonumber\\&\leq  C\int_{0}^{t}{(\|\sqrt{\rho}\partial_\tau u\|_{L^{2}}^{1-\frac{2n(p-2)}{(n+2)p}}
\|\Lambda^{\frac{1}{2}+\frac{n}{4}}\partial_\tau u\|_{L^{2}}^{\frac{2n(p-2)}{(n+2)p}}+\|\Lambda^{\frac{1}{2}+\frac{n}{4}}u(\tau)\|_{L^{2}}^{2}
+ \|\Lambda^{1+\frac{n}{2}}u(\tau)\|_{L^{2}}^{2}) \,d\tau}
\nonumber\\&=
C\int_{0}^{1}{\tau^{-\frac{1}{2}}(\tau^{\frac{1}{2}}\|\sqrt{\rho}\partial_\tau u\|_{L^{2}})^{1-\frac{2n(p-2)}{(n+2)p}}
(\tau^{\frac{1}{2}}\|\Lambda^{\frac{1}{2}+\frac{n}{4}}\partial_\tau u\|_{L^{2}})^{\frac{2n(p-2)}{(n+2)p}}) \,d\tau}
\nonumber\\& \quad+C\int_{1}^{t}{e^{-\frac{\gamma\tau}{2}}(e^{\frac{\gamma\tau}{2}}
\|\sqrt{\rho}\partial_\tau u(\tau)\|_{L^{2}} +
e^{\frac{\gamma\tau}{2}}\|\Lambda^{\frac{1}{2}+\frac{n}{4}}\partial_\tau u(\tau)\|_{L^{2}}) \,d\tau}
\nonumber\\& \quad+C\int_{0}^{t}{e^{-\gamma\tau}(e^{\gamma\tau}\|\Lambda^{\frac{1}{2}+\frac{n}{4}}u(\tau)\|_{L^{2}}^{2}+
e^{\gamma\tau}\|\Lambda^{1+\frac{n}{2}}u(\tau)\|_{L^{2}}^{2}) \,d\tau}
\nonumber\\&\leq  C\int_{0}^{1}{\tau^{-\frac{1}{2}}
(\tau^{\frac{1}{2}}\|\Lambda^{\frac{1}{2}+\frac{n}{4}}\partial_\tau u\|_{L^{2}})^{\frac{2n(p-2)}{(n+2)p}}) \,d\tau}+\widetilde{C}\int_{0}^{t}{e^{-\gamma\tau} \,d\tau}
\nonumber\\&\leq \widetilde{C} + C\left(\int_{0}^{1}{\tau\|\Lambda^{\frac{1}{2}+\frac{n}{4}}\partial_\tau u\|_{L^{2}}^{2} \,d\tau}\right)^{\frac{n(p-2)}{(n+2)p}}
\left(\int_{0}^{1}{\tau^{-\frac{(n+2)p}{4(n+p)}}  \,d\tau}\right)^{1-\frac{n(p-2)}{(n+2)p}}
\nonumber\\&\leq  \widetilde{C}.
\end{align}
Applying the $L^{p}$-estimate to (\ref{521tt005}) yields
\begin{align}
\|\Lambda^{1+\frac{n}{2}}u\|_{L^{p}}\leq C\|\rho\partial_tu\|_{L^{p}}
+C\|\rho u\cdot\nabla u\|_{L^{p}},\nonumber
\end{align}
which leads to
\begin{align}\label{keyusef3}
\|\nabla u\|_{L^{\infty}}&\leq C\|\nabla u\|_{L^{\frac{4n}{n+2}}}^{1-\frac{(n+2)p}{(3n+2)p-4n}}
\|\Lambda^{1+\frac{n}{2}}u\|_{L^{p}}^{\frac{(n+2)p}{(3n+2)p-4n}}\nonumber\\
&\leq  C\|\Lambda^{\frac{1}{2}+\frac{n}{4}} u\|_{L^{2}}^{1-\frac{(n+2)p}{(3n+2)p-4n}}
\|\Lambda^{1+\frac{n}{2}}u\|_{L^{p}}^{\frac{(n+2)p}{(3n+2)p-4n}}
\nonumber\\
&\leq  C\|\Lambda^{\frac{1}{2}+\frac{n}{4}} u\|_{L^{2}}+C
\|\Lambda^{1+\frac{n}{2}}u\|_{L^{p}}
\nonumber\\
&\leq  C\|\Lambda^{\frac{1}{2}+\frac{n}{4}} u\|_{L^{2}}+C\|\rho\partial_tu\|_{L^{p}}
+C\|\rho u\cdot\nabla u\|_{L^{p}}.
\end{align}
Thanks to (\ref{addt030t001}), (\ref{521tt010}) and (\ref{keyusef3}), we immediately obtain
\begin{eqnarray} \label{521tt011}
\int_{0}^{t}{\|\nabla u(\tau)
\|_{L^{\infty}} \,d\tau} \leq \widetilde{C_{1}}.
\end{eqnarray}
Since $\rho$ satisfies
$\partial_t \rho+u\cdot\nabla\rho=0,$
direct computations yield
$$\frac{d}{dt}\|\nabla\rho(t)\|_{L^{\frac{4n}{n+6}}}\leq \|\nabla u\|_{L^{\infty}}\|\nabla\rho(t)\|_{L^{\frac{4n}{n+6}}},\quad
 \frac{d}{dt}\|\nabla\rho(t)\|_{L^{2}}\leq \|\nabla u\|_{L^{\infty}}\|\nabla\rho(t)\|_{L^{2}},$$
The Gronwall inequality and (\ref{521tt011}) ensure that
\begin{eqnarray}
\|\nabla\rho(t)\|_{L^{\frac{4n}{n+6}}}\leq \|\nabla\rho_{0}\|_{L^{\frac{4n}{n+6}}}\exp\left[\int_{0}^{t}{\|\nabla u(\tau)
\|_{L^{\infty}} \,d\tau}\right]\leq \widetilde{C_{1}}\|\nabla\rho_{0}\|_{L^{\frac{4n}{n+6}}},\nonumber
\end{eqnarray}
\begin{eqnarray}
\|\nabla\rho(t)\|_{L^{2}}\leq \|\nabla\rho_{0}\|_{L^{2}}\exp\left[\int_{0}^{t}{\|\nabla u(\tau)
\|_{L^{\infty}} \,d\tau}\right]\leq \widetilde{C_{1}}\|\nabla\rho_{0}\|_{L^{2}}.\nonumber
\end{eqnarray}
We thus complete the proof of Lemma \ref{INDL224}.
\end{proof}

\subsection{Time-weighted estimates and exponential
decay of $\|\Lambda^{\frac{1}{2}+\frac{n}{4}}\partial_{t}u(t)
\|_{L^{2}}^{2}$}

\begin{lemma}\label{INDL226}
Under the assumptions of Theorem \ref{Th1}, the   solution $(\rho,u)$
of the system (\ref{INDNSE}) admits the following bound for any $t\geq0$,
\begin{eqnarray} \label{t5529b01}
t^{2}\|\Lambda^{\frac{1}{2}+\frac{n}{4}}\partial_{t}u(t)
\|_{L^{2}}^{2} +\int_{0}^{t}{\tau^{2}\|\sqrt{\rho}\partial_{\tau\tau}u(\tau)\|_{L^{2}}^{2}\,d\tau} \leq \widetilde{C}.
\end{eqnarray}
Moreover, for any $t\geq1$, we have
\begin{eqnarray} \label{dsvv54t225}
e^{\gamma t}\|\Lambda^{\frac{1}{2}+\frac{n}{4}}\partial_{t}u(t)
\|_{L^{2}}^{2} +\int_{1}^{t}{e^{\gamma \tau}\|\sqrt{\rho}\partial_{\tau\tau}u(\tau)\|_{L^{2}}^{2}\,d\tau} \leq \widetilde{C},
\end{eqnarray}
where $\widetilde{C}$ depends only on $\|\rho_{0}\|_{L^{\frac{2n}{n+2}}}$, $\|\rho_{0}\|_{L^{\infty}}$, $\|\nabla\rho_{0}\|_{L^{2}}$, $\|\sqrt{\rho_{0}}u_{0}\|_{L^{2}}$ and $\|\Lambda^{\frac{1}{2}+\frac{n}{4}}u_{0}\|_{L^{2}}$.
\end{lemma}

\begin{proof}
Multiplying (\ref{521tt008}) by $\partial_{tt}u$ and integrating by parts imply that
\begin{align}& \label{dsvv54t226}
 \frac{1}{2}\frac{d}{dt}\|\Lambda^{\frac{1}{2}+\frac{n}{4}}\partial_{t}u(t)
 \|_{L^{2}}^{2} +\|\sqrt{\rho}\partial_{tt}u\|_{L^{2}}^{2}\nonumber\\&=
-\int_{\mathbb{R}^{n}}\partial_{t}\rho\partial_{t}u\cdot \partial_{tt}u\,dx-\int_{\mathbb{R}^{n}}\partial_{t}(\rho u)\cdot\nabla u\cdot \partial_{tt}u\,dx-\int_{\mathbb{R}^{n}}\rho u\cdot\nabla \partial_{t}u\cdot \partial_{tt}u\,dx\nonumber\\
 &= -\int_{\mathbb{R}^{n}}\partial_{t}\rho\partial_{t}u\cdot \partial_{tt}u\,dx-\int_{\mathbb{R}^{n}}\partial_{t}\rho u\cdot\nabla u\cdot \partial_{tt}u\,dx-\int_{\mathbb{R}^{n}} \rho \partial_{t}u\cdot\nabla u\cdot \partial_{tt}u\,dx\nonumber\\& \quad-\int_{\mathbb{R}^{n}}\rho u\cdot\nabla \partial_{t}u\cdot \partial_{tt}u\,dx
 \nonumber\\
 &:= H_{1}+H_{2}+H_{3}+H_{4}.
\end{align}
We first bound $H_{3}$ and $H_{4}$ as
\begin{align}
|H_{3}|+|H_{4}|&\leq  C\|\sqrt{\rho}\|_{L^{\infty}}\|\sqrt{\rho}\partial_{tt}u\|_{L^{2}}
(\|\partial_{t}u\|_{L^{\frac{4n}{n-2}}}\|\nabla u\|_{L^{\frac{4n}{n+2}}}+\| u\|_{L^{\frac{4n}{n-2}}}\|\partial_{t}\nabla u\|_{L^{\frac{4n}{n+2}}})\nonumber\\
&\leq  C\|\sqrt{\rho_{0}}\|_{L^{\infty}}\|\sqrt{\rho}\partial_{tt}u\|_{L^{2}}
\|\Lambda^{\frac{1}{2}+\frac{n}{4}}\partial_{t}u\|_{L^{2}}\|\Lambda^{\frac{1}{2}+\frac{n}{4}} u\|_{L^{2}}
\nonumber\\
&\leq \frac{1}{2}\|\sqrt{\rho}\partial_{tt}u\|_{L^{2}}^{2}+
C\|\Lambda^{\frac{1}{2}+\frac{n}{4}} u\|_{L^{2}}^{2}\|\Lambda^{\frac{1}{2}+\frac{n}{4}}\partial_{t}u\|_{L^{2}}^{2}.
\nonumber
\end{align}
We rewrite $H_{1}$ as follows
\begin{align}\label{w34cvty76}
H_{1} &= -\frac{1}{2}\int_{\mathbb{R}^{n}}\partial_{t}\rho \partial_{t}|\partial_{t}u|^{2}\,dx\nonumber\\
&=  -\frac{1}{2}\frac{d}{dt}\int_{\mathbb{R}^{n}}\partial_{t}\rho |\partial_{t}u|^{2}\,dx+\frac{1}{2} \int_{\mathbb{R}^{n}}\partial_{tt}\rho |\partial_{t}u|^{2}\,dx\nonumber\\
&=  -\frac{1}{2}\frac{d}{dt}\int_{\mathbb{R}^{n}}\partial_{t}\rho |\partial_{t}u|^{2}\,dx-\frac{1}{2} \int_{\mathbb{R}^{n}}\partial_{t}{\rm div}(\rho u) |\partial_{t}u|^{2}\,dx\nonumber\\
&=  -\frac{1}{2}\frac{d}{dt}\int_{\mathbb{R}^{n}}\partial_{t}\rho |\partial_{t}u|^{2}\,dx+\int_{\mathbb{R}^{n}}\partial_{t} (\rho u_{i} ) \partial_{t}u\cdot\partial_{t}\partial_{i}u\,dx
\nonumber\\
&= -\frac{1}{2}\frac{d}{dt}\int_{\mathbb{R}^{n}}\partial_{t}\rho |\partial_{t}u|^{2}\,dx+\int_{\mathbb{R}^{n}}\partial_{t}\rho u_{i} \partial_{t}u\cdot\partial_{t}\partial_{i}u\,dx\nonumber\\
& \quad+\int_{\mathbb{R}^{n}}\rho \partial_{t}u_{i} \partial_{t}u\cdot\partial_{t}\partial_{i}u\,dx.
\end{align}
By the H$\rm\ddot{o}$lder inequality and the embedding inequality, we have
\begin{align}
\int_{\mathbb{R}^{n}}\rho \partial_{t}u_{i} \partial_{t}u\cdot\partial_{t}\partial_{i}u\,dx&\leq
C\|\sqrt{\rho}\|_{L^{\infty}}\|\sqrt{\rho}\partial_{t}u\|_{L^{2}}\|\partial_{t}u\|_{L^{\frac{4n}{n-2}}}\|\partial_{t}\nabla u\|_{L^{\frac{4n}{n+2}}}\nonumber\\
&\leq
C\|\sqrt{\rho}\partial_{t}u\|_{L^{2}}\|\Lambda^{\frac{1}{2}+\frac{n}{4}}\partial_{t}u\|_{L^{2}}^{2}.\nonumber
\end{align}
Similarly, using $\partial_{t}\rho=-u\cdot\nabla\rho$ gives
\begin{align}
\int_{\mathbb{R}^{n}}\partial_{t}\rho u_{i} \partial_{t}u\cdot\partial_{t}\partial_{i}u\,dx&\leq
C\|u\cdot\nabla\rho\|_{L^{2}}\|u\|_{L^{\infty}}
\|\partial_{t}u\|_{L^{\frac{4n}{n-2}}} \|\partial_{t}\nabla u\|_{L^{\frac{4n}{n+2}}}\nonumber\\
&\leq
C\|\nabla\rho\|_{L^{2}}\|u\|_{L^{\infty}}^{2}
\|\Lambda^{\frac{1}{2}+\frac{n}{4}}\partial_{t}u\|_{L^{2}}^{2}\nonumber\\
&\leq
C\|\nabla\rho_{0}\|_{L^{2}}\|\Lambda^{\frac{1}{2}+\frac{n}{4}}u\|_{L^{2}}
^{\frac{8}{n+2}}\|\Lambda^{1+\frac{n}{2}}u\|_{L^{2}}
^{\frac{2(n-2)}{n+2}}
\|\Lambda^{\frac{1}{2}+\frac{n}{4}}\partial_{t}u\|_{L^{2}}^{2}\nonumber\\
&\leq
C\|\Lambda^{\frac{1}{2}+\frac{n}{4}}u\|_{L^{2}}
^{\frac{8}{n+2}}\|\Lambda^{1+\frac{n}{2}}u\|_{L^{2}}
^{\frac{2(n-2)}{n+2}}
\|\Lambda^{\frac{1}{2}+\frac{n}{4}}\partial_{t}u\|_{L^{2}}^{2}.\nonumber
\end{align}
We obtain that
\begin{align}
H_{1} &\leq -\frac{1}{2}\frac{d}{dt}\int_{\mathbb{R}^{n}}\partial_{t}\rho |\partial_{t}u|^{2}\,dx+ C(\|\sqrt{\rho}\partial_{t}u\|_{L^{2}}+
\|\Lambda^{\frac{1}{2}+\frac{n}{4}}u\|_{L^{2}}
^{\frac{8}{n+2}}\|\Lambda^{1+\frac{n}{2}}u\|_{L^{2}}
^{\frac{2(n-2)}{n+2}})\nonumber\\& \quad\times
\|\Lambda^{\frac{1}{2}+\frac{n}{4}}\partial_{t}u\|_{L^{2}}^{2}.\nonumber
\end{align}
Thanks to $\partial_{t}\rho=-{\rm div}(\rho u)$, one obtains
\begin{align}\label{w34cvty78}
H_{2}&= -\frac{d}{dt}\int_{\mathbb{R}^{n}}\partial_{t}\rho u\cdot\nabla u\cdot \partial_{t}u\,dx+\int_{\mathbb{R}^{n}}\partial_{tt}\rho u\cdot\nabla u\cdot \partial_{t}u\,dx+\int_{\mathbb{R}^{n}}\partial_{t}\rho \partial_{t}(u\cdot\nabla u)\cdot \partial_{t}u\,dx\nonumber\\
&= -\frac{d}{dt}\int_{\mathbb{R}^{n}}\partial_{t}\rho u\cdot\nabla u\cdot \partial_{t}u\,dx+\int_{\mathbb{R}^{n}}\partial_{t}(\rho u_{i})\partial_{i}(u\cdot\nabla u)\cdot \partial_{t}u\,dx
\nonumber\\& \quad+\int_{\mathbb{R}^{n}}\partial_{t}(\rho u_{i})u\cdot\nabla u\cdot \partial_{t}\partial_{i}u\,dx
+\int_{\mathbb{R}^{n}}\partial_{t}\rho \partial_{t}(u\cdot\nabla u)\cdot \partial_{t}u\,dx
\nonumber\\
&= -\frac{d}{dt}\int_{\mathbb{R}^{n}}\partial_{t}\rho u\cdot\nabla u\cdot \partial_{t}u\,dx+\int_{\mathbb{R}^{n}}\rho \partial_{t}u_{i}[\partial_{i}(u\cdot\nabla u)\cdot \partial_{t}u+u\cdot\nabla u\cdot \partial_{t}\partial_{i}u]\,dx
\nonumber\\& \quad+\int_{\mathbb{R}^{n}}\partial_{t}\rho[ u_{i}u\cdot\nabla u\cdot \partial_{t}\partial_{i}u+u_{i}\partial_{i}(u\cdot\nabla u)\cdot \partial_{t}u]\,dx
+\int_{\mathbb{R}^{n}}\partial_{t}\rho \partial_{t}(u\cdot\nabla u)\cdot \partial_{t}u\,dx\nonumber\\
&= -\frac{d}{dt}\int_{\mathbb{R}^{n}}\partial_{t}\rho u\cdot\nabla u\cdot \partial_{t}u\,dx+H_{21}+H_{22}+H_{23}.
\end{align}
It follows from the H$\rm\ddot{o}$lder inequality and the interpolation inequalities that
\begin{align}
H_{21} &\leq C\|\rho\|_{L^{\infty}}\|\partial_{t}u\|_{L^{\frac{4n}{n-2}}}\|\nabla u\nabla u\|_{L^{\frac{2n}{n+2}}} \|\partial_{t}u\|_{L^{\frac{4n}{n-2}}}\nonumber\\& \quad+
C\|\sqrt{\rho}\|_{L^{\infty}}\|\partial_{t}u\|_{L^{\frac{4n}{n-2}}}\|\sqrt{\rho} u\|_{L^{2}}\|\nabla^{2}u\|_{L^{n}}\|\partial_{t}u\|_{L^{\frac{4n}{n-2}}}
\nonumber\\& \quad+
C\|\rho\|_{L^{\infty}}\|\partial_{t}u\|_{L^{\frac{4n}{n-2}}}\|u\|_{L^{\frac{4n}{n-2}}} \|\nabla u\|_{L^{\frac{4n}{n+2}}}\|\nabla\partial_{t}u\|_{L^{\frac{4n}{n+2}}}
\nonumber\\&\leq
C(\|\sqrt{\rho} u\|_{L^{2}}\|\Lambda^{1+\frac{n}{2}}u\|_{L^{2}}+
\|\Lambda^{\frac{1}{2}+\frac{n}{4}}u\|_{L^{2}}
^{2})
\|\Lambda^{\frac{1}{2}+\frac{n}{4}}\partial_{t}u\|_{L^{2}}^{2},\nonumber
\end{align}
\begin{align}
H_{22} &\leq C\|\partial_{t}\rho\|_{L^{2}}\|u\|_{L^{\infty}}^{2}\|\nabla u\|_{L^{\frac{4n}{n-2}}} \|\nabla\partial_{t}u\|_{L^{\frac{4n}{n+2}}}\nonumber\\& \quad+
C\|\partial_{t}\rho\|_{L^{2}}
\|u\|_{L^{\infty}}(\|\nabla u \nabla u\|_{L^{\frac{4n}{n+2}}}+\| u \nabla^{2} u\|_{L^{\frac{4n}{n+2}}})\|\partial_{t}u\|_{L^{\frac{4n}{n-2}}}
\nonumber\\&\leq C\|u\cdot\nabla\rho\|_{L^{2}}\|u\|_{L^{\infty}}^{2}
\|\Lambda^{\frac{3}{2}+\frac{n}{4}} u\|_{L^{2}} \|\Lambda^{\frac{1}{2}+\frac{n}{4}}\partial_{t}u\|_{L^{2}}
\nonumber\\& \quad+C\|u\cdot\nabla\rho\|_{L^{2}}\|u\|_{L^{\infty}}
\|\nabla u\|_{L^{\frac{8n}{n+2}}}^{2} \|\Lambda^{\frac{1}{2}+\frac{n}{4}}\partial_{t}u\|_{L^{2}}
\nonumber\\&\leq C\|\nabla\rho\|_{L^{2}}\|u\|_{L^{\infty}}^{3}
\|\Lambda^{\frac{3}{2}+\frac{n}{4}} u\|_{L^{2}} \|\Lambda^{\frac{1}{2}+\frac{n}{4}}\partial_{t}u\|_{L^{2}}\nonumber\\& \quad
+C\|\nabla\rho\|_{L^{2}}
\|u\|_{L^{\infty}}^{2}
\|\nabla u\|_{L^{\frac{8n}{n+2}}}^{2} \|\Lambda^{\frac{1}{2}+\frac{n}{4}}\partial_{t}u\|_{L^{2}}
\nonumber\\&\leq C\|\nabla\rho_{0}\|_{L^{2}}
(\|\Lambda^{\frac{1}{2}+\frac{n}{4}}u\|_{L^{2}}^{\frac{12}{n+2}}
\|\Lambda^{1+\frac{n}{2}}u\|_{L^{2}}^{\frac{3(n-2)}{n+2}})
(\|\Lambda^{\frac{1}{2}+\frac{n}{4}}u\|_{L^{2}}^{\frac{n-2}{n+2}}
\|\Lambda^{1+\frac{n}{2}}u\|_{L^{2}}^{\frac{4}{n+2}}) \|\Lambda^{\frac{1}{2}+\frac{n}{4}}\partial_{t}u\|_{L^{2}}
\nonumber\\& \quad+C\|\nabla\rho_{0}\|_{L^{2}}
(\|\Lambda^{\frac{1}{2}+\frac{n}{4}}u\|_{L^{2}}^{\frac{8}{n+2}}
\|\Lambda^{1+\frac{n}{2}}u\|_{L^{2}}^{\frac{2(n-2)}{n+2}})
(\|\Lambda^{\frac{1}{2}+\frac{n}{4}}u\|_{L^{2}}
\|\Lambda^{1+\frac{n}{2}}u\|_{L^{2}}) \|\Lambda^{\frac{1}{2}+\frac{n}{4}}\partial_{t}u\|_{L^{2}}
\nonumber\\&\leq
C\|\Lambda^{\frac{1}{2}+\frac{n}{4}}u\|_{L^{2}}^{\frac{n+10}{n+2}}
\|\Lambda^{1+\frac{n}{2}}u\|_{L^{2}}^{\frac{3n-2}{n+2}}
\|\Lambda^{\frac{1}{2}+\frac{n}{4}}\partial_{t}u\|_{L^{2}}.\nonumber
\end{align}
Similarly, we have
\begin{align}
H_{23} &\leq C\|\partial_{t}\rho\|_{L^{2}}\|\partial_{t}u\|_{L^{\frac{4n}{n-2}}}\|\nabla u\|_{L^{n}} \|\partial_{t}u\|_{L^{\frac{4n}{n-2}}}\nonumber\\& \quad+
C\|\partial_{t}\rho\|_{L^{2}}
\|u\|_{L^{\infty}}\|\nabla\partial_{t}u\|_{L^{\frac{4n}{n+2}}}\|\partial_{t}u\|_{L^{\frac{4n}{n-2}}}
\nonumber\\&\leq C\|u\cdot\nabla\rho\|_{L^{2}}(\|\nabla u\|_{L^{n}}+\| u\|_{L^{\infty}}) \|\Lambda^{\frac{1}{2}+\frac{n}{4}}\partial_{t}u\|_{L^{2}}^{2}
\nonumber\\&\leq C\|\nabla\rho\|_{L^{2}}(\|\nabla u\|_{L^{n}}^{2}+\| u\|_{L^{\infty}}^{2}) \|\Lambda^{\frac{1}{2}+\frac{n}{4}}\partial_{t}u\|_{L^{2}}^{2}
\nonumber\\&\leq C\|\nabla\rho_{0}\|_{L^{2}}
(\|\Lambda^{\frac{1}{2}+\frac{n}{4}}u\|_{L^{2}}^{\frac{8}{n+2}}
\|\Lambda^{1+\frac{n}{2}}u\|_{L^{2}}^{\frac{2(n-2)}{n+2}}) \|\Lambda^{\frac{1}{2}+\frac{n}{4}}\partial_{t}u\|_{L^{2}}^{2}
\nonumber\\&\leq
C\|\Lambda^{\frac{1}{2}+\frac{n}{4}}u\|_{L^{2}}^{\frac{8}{n+2}}
\|\Lambda^{1+\frac{n}{2}}u\|_{L^{2}}^{\frac{2(n-2)}{n+2}} \|\Lambda^{\frac{1}{2}+\frac{n}{4}}\partial_{t}u\|_{L^{2}}^{2}.\nonumber
\end{align}
Therefore, $H_{2}$ admits the following bound
\begin{align}
H_{2}&\leq -\frac{d}{dt}\int_{\mathbb{R}^{n}}\partial_{t}\rho u\cdot\nabla u\cdot \partial_{t}u\,dx+C(\|\sqrt{\rho} u\|_{L^{2}}\|\Lambda^{1+\frac{n}{2}}u\|_{L^{2}}+
\|\Lambda^{\frac{1}{2}+\frac{n}{4}}u\|_{L^{2}}
^{2}\nonumber\\& \quad+\|\Lambda^{\frac{1}{2}+\frac{n}{4}}u\|_{L^{2}}^{\frac{8}{n+2}}
\|\Lambda^{1+\frac{n}{2}}u\|_{L^{2}}^{\frac{2(n-2)}{n+2}})
\|\Lambda^{\frac{1}{2}+\frac{n}{4}}\partial_{t}u\|_{L^{2}}^{2}\nonumber\\& \quad
+C\|\Lambda^{\frac{1}{2}+\frac{n}{4}}u\|_{L^{2}}^{\frac{n+10}{n+2}}
\|\Lambda^{1+\frac{n}{2}}u\|_{L^{2}}^{\frac{3n-2}{n+2}}
\|\Lambda^{\frac{1}{2}+\frac{n}{4}}\partial_{t}u\|_{L^{2}}.\nonumber
\end{align}
We finally get by collecting all the above estimates
\begin{align}\label{dsvv54t227}&
\frac{d}{dt}\left(\|\Lambda^{\frac{1}{2}+\frac{n}{4}}\partial_{t}u(t)
 \|_{L^{2}}^{2}+\phi(t)\right) +\|\sqrt{\rho}\partial_{tt}u\|_{L^{2}}^{2}\nonumber\\&\leq  A(t)\|\Lambda^{\frac{1}{2}+\frac{n}{4}}\partial_{t}u\|_{L^{2}}+
B(t)\|\Lambda^{\frac{1}{2}+\frac{n}{4}}\partial_{t}u\|_{L^{2}}^{2},
\end{align}
where
$$\phi(t):=-\frac{1}{2} \int_{\mathbb{R}^{n}}\partial_{t}\rho |\partial_{t}u|^{2}\,dx- \int_{\mathbb{R}^{n}}\partial_{t}\rho u\cdot\nabla u\cdot \partial_{t}u\,dx,$$
$$A(t):=C\|\Lambda^{\frac{1}{2}+\frac{n}{4}}u(t)\|_{L^{2}}^{\frac{n+10}{n+2}}
\|\Lambda^{1+\frac{n}{2}}u(t)\|_{L^{2}}^{\frac{3n-2}{n+2}},$$
\begin{align*}
B(t):=\ &C(\|\sqrt{\rho} u(t)\|_{L^{2}}\|\Lambda^{1+\frac{n}{2}}u(t)\|_{L^{2}}+
\|\Lambda^{\frac{1}{2}+\frac{n}{4}}u(t)\|_{L^{2}}
^{2}\\
&+\|\Lambda^{\frac{1}{2}+\frac{n}{4}}u(t)\|_{L^{2}}^{\frac{8}{n+2}}
\|\Lambda^{1+\frac{n}{2}}u(t)\|_{L^{2}}^{\frac{2(n-2)}{n+2}}
+\|\sqrt{\rho}\partial_{t}u\|_{L^{2}}).
\end{align*}
Hence, in view of $\partial_{t}\rho=-{\rm div}(\rho u)=-u\cdot\nabla\rho$ and the H$\rm\ddot{o}$lder inequality along with the embedding inequality, we deduce
\begin{align}\label{dfr56bn28}
\left|\phi(t)\right|&= \left|\frac{1}{2} \int_{\mathbb{R}^{n}}{\rm div}(\rho u) |\partial_{t}u|^{2}\,dx- \int_{\mathbb{R}^{n}}\partial_{t}\rho u\cdot\nabla u\cdot \partial_{t}u\,dx \right|\nonumber\\
&= \left|- \int_{\mathbb{R}^{n}} \rho u_{i}\partial_{t}u\cdot\partial_{t}\partial_{i}u\,dx- \int_{\mathbb{R}^{n}}\partial_{t}\rho u\cdot\nabla u\cdot \partial_{t}u\,dx\right|
\nonumber\\
&\leq C\|\sqrt{\rho}\|_{L^{\infty}}\|\sqrt{\rho}\partial_{t}u\|_{L^{2}}\| u\|_{L^{\frac{4n}{n-2}}} \|\nabla\partial_{t}u\|_{L^{\frac{4n}{n+2}}}\nonumber\\&\quad
+C\|\partial_{t}\rho\|_{L^{2}}\|u\|_{L^{\infty}}\|\nabla u\|_{L^{\frac{4n}{n+2}}} \|\partial_{t}u\|_{L^{\frac{4n}{n-2}}}
\nonumber\\
&\leq C\|\sqrt{\rho_{0}}\|_{L^{\infty}}\|\sqrt{\rho}\partial_{t}u\|_{L^{2}}\| \Lambda^{\frac{1}{2}+\frac{n}{4}}u\|_{L^{2}} \|\Lambda^{\frac{1}{2}+\frac{n}{4}}\partial_{t}u\|_{L^{2}}\nonumber\\&\quad
+C\|\nabla\rho\|_{L^{2}}\|u\|_{L^{\infty}}^{2}\|\Lambda^{\frac{1}{2}+\frac{n}{4}}
u\|_{L^{2}} \|\Lambda^{\frac{1}{2}+\frac{n}{4}}\partial_{t}u\|_{L^{2}}\nonumber\\
&\leq C\|\sqrt{\rho_{0}}\|_{L^{\infty}}\|\sqrt{\rho}\partial_{t}u\|_{L^{2}}\| \Lambda^{\frac{1}{2}+\frac{n}{4}}u\|_{L^{2}} \|\Lambda^{\frac{1}{2}+\frac{n}{4}}\partial_{t}u\|_{L^{2}}\nonumber\\&\quad
+C\|\nabla\rho_{0}\|_{L^{2}}(\|\Lambda^{\frac{1}{2}+\frac{n}{4}}
u\|_{L^{2}}^{\frac{8}{n+2}}
\|\Lambda^{1+\frac{n}{2}}u\|_{L^{2}}^{\frac{2(n-2)}{n+2}})
\|\Lambda^{\frac{1}{2}+\frac{n}{4}}
u\|_{L^{2}} \|\Lambda^{\frac{1}{2}+\frac{n}{4}}\partial_{t}u\|_{L^{2}}
\nonumber\\
&\leq \frac{1}{2}\|\Lambda^{\frac{1}{2}+\frac{n}{4}}\partial_{t}u\|_{L^{2}}^{2}+
C\|\sqrt{\rho}\partial_{t}u\|_{L^{2}}^{2}\| \Lambda^{\frac{1}{2}+\frac{n}{4}}u\|_{L^{2}}^{2}\nonumber\\&\quad+
C\|\Lambda^{\frac{1}{2}+\frac{n}{4}}u\|_{L^{2}}^{\frac{2n+20}{n+2}}
\|\Lambda^{1+\frac{n}{2}}u\|_{L^{2}}^{\frac{4(n-2)}{n+2}}.
\end{align}
We first get from (\ref{dsvv54t227}) that
\begin{align}\label{xcrtinm01}&
\frac{d}{dt}\left(t^{2}\|\Lambda^{\frac{1}{2}+\frac{n}{4}}\partial_{t}u(t)
 \|_{L^{2}}^{2}+t^{2}\phi(t)\right) +t^{2}\|\sqrt{\rho}\partial_{tt}u\|_{L^{2}}^{2}\nonumber\\&\leq  2t\|\Lambda^{\frac{1}{2}+\frac{n}{4}}\partial_{t}u(t)
 \|_{L^{2}}^{2}+2t\phi(t)+t^{2}A(t)\|\Lambda^{\frac{1}{2}+\frac{n}{4}}\partial_{t}u\|_{L^{2}} +
B(t)t^{2}\|\Lambda^{\frac{1}{2}+\frac{n}{4}}\partial_{t}u\|_{L^{2}}^{2}.
\end{align}
By (\ref{addt030t001}), (\ref{addt030t003}) and (\ref{addok5251}), we conclude
\begin{align}\label{xcrtinm02}&
\int_{0}^{t}\tau^{2}A(\tau)\|\Lambda^{\frac{1}{2}+\frac{n}{4}}
\partial_{\tau}u(\tau)\|_{L^{2}}\,d\tau \nonumber\\&\leq  C \left(\int_{0}^{t}\tau^{3}A^{2}(\tau) \,d\tau\right)^{\frac{1}{2}} \left(\int_{0}^{t}\tau\|\Lambda^{\frac{1}{2}+\frac{n}{4}}
\partial_{\tau}u(\tau)\|_{L^{2}}^{2}\,d\tau\right)^{\frac{1}{2}}\nonumber\\
&\leq  C \left(\int_{0}^{t}\tau^{3}\|\Lambda^{\frac{1}{2}+\frac{n}{4}}u(\tau)\|_{L^{2}}
^{\frac{2(n+10)}{n+2}}
\|\Lambda^{1+\frac{n}{2}}u(\tau)\|_{L^{2}}^{\frac{2(3n-2)}{n+2}} \,d\tau\right)^{\frac{1}{2}}
\nonumber\\
&=  C \left(\int_{0}^{t}\tau^{\frac{8}{n+2}}e^{-\frac{(n+10)\gamma\tau}{n+2}}(e^{\gamma\tau}
\|\Lambda^{\frac{1}{2}+\frac{n}{4}}u(\tau)\|_{L^{2}}^{2})
^{\frac{n+10}{n+2}}
(\tau\|\Lambda^{1+\frac{n}{2}}u(\tau)\|_{L^{2}}^{2})^{\frac{3n-2}{n+2}} \,d\tau\right)^{\frac{1}{2}}
\nonumber\\
&\leq  C \left(\int_{0}^{t}\tau^{\frac{8}{n+2}}e^{-\frac{(n+10)\gamma\tau}{n+2}} \,d\tau\right)^{\frac{1}{2}}\leq \widetilde{C},
\end{align}
where and in what follows, we use the following facts: for any $\sigma_{1}\geq0,\,\sigma_{2}>0$,
$$\int_{0}^{\infty}\eta^{\sigma_{1}}e^{-\sigma_{2}\eta}\,d\eta<\infty\quad \mbox{and}\quad \tau^{\sigma_{1}}e^{-\sigma_{2}\tau}<\infty,\ \forall\,\tau\geq0.$$
Noticing the following estimate
\begin{align}
\int_{0}^{t}\|\sqrt{\rho}\partial_{t}u(\tau)\|_{L^{2}}\,d\tau&=
\int_{0}^{t}e^{-\gamma \frac{\tau}{2}}e^{\gamma \frac{\tau}{2}}\|\sqrt{\rho}\partial_{t}u(\tau)\|_{L^{2}}\,d\tau
\nonumber\\&\leq
\left(\int_{0}^{t}e^{-\gamma \tau} \right)^{\frac{1}{2}}\left(\int_{0}^{t}e^{\gamma\tau}
\|\sqrt{\rho}\partial_{t}u(\tau)\|_{L^{2}}^{2}
\,d\tau\right)^{\frac{1}{2}} \leq
\widetilde{C}
\end{align}
and using the argument in dealing with (\ref{xcrtinm02}), we show that
\begin{eqnarray}\label{xcrtinm03}
\int_{0}^{t}
B(\tau)\,d\tau\leq \widetilde{C}.
\end{eqnarray}
According to (\ref{addt030t001}), (\ref{addt030t003}) and (\ref{addok5251}) again, one deduces from (\ref{dfr56bn28}) that
\begin{align}\label{xcrtinm04}\int_{0}^{t}
\tau\phi(\tau)\,d\tau&\leq \frac{1}{2}\int_{0}^{t}
\tau\|\Lambda^{\frac{1}{2}+\frac{n}{4}}\partial_{\tau}u(\tau)\|_{L^{2}}^{2}\,d\tau
+C\int_{0}^{t}
\tau\|\sqrt{\rho}\partial_{\tau}u(\tau)\|_{L^{2}}^{2}\| \Lambda^{\frac{1}{2}+\frac{n}{4}}u(\tau)\|_{L^{2}}^{2}\,d\tau\nonumber\\& \quad
+C\int_{0}^{t}
\tau\|\Lambda^{\frac{1}{2}+\frac{n}{4}}u(\tau)\|_{L^{2}}^{\frac{2n+20}{n+2}}
\|\Lambda^{1+\frac{n}{2}}u(\tau)\|_{L^{2}}^{\frac{4(n-2)}{n+2}}\,d\tau\nonumber\\
&\leq \widetilde{C}+C\int_{0}^{t}
\tau e^{-\frac{(n+10)\gamma\tau}{n+2}}(e^{\gamma\tau}\|\Lambda^{\frac{1}{2}+\frac{n}{4}}
u(\tau)\|_{L^{2}}^{2})
^{\frac{n+10}{n+2}}
\|\Lambda^{1+\frac{n}{2}}u(\tau)\|_{L^{2}}^{\frac{4(n-2)}{n+2}}\,d\tau
\nonumber\\
&\leq  \widetilde{C}+C\int_{0}^{t}
\tau e^{-\frac{(n+10)\gamma\tau}{n+2}}
\|\Lambda^{1+\frac{n}{2}}u(\tau)\|_{L^{2}}^{\frac{4(n-2)}{n+2}}\,d\tau
\nonumber\\
&\leq \widetilde{C}+C\chi_{\{n\geq 6\}}\int_{0}^{t}
\tau^{\frac{8}{n+2}} e^{-\frac{(n+10)\gamma\tau}{n+2}}
(\tau\|\Lambda^{1+\frac{n}{2}}u(\tau)\|_{L^{2}}^{2})^{\frac{n-6}{n+2}}
\|\Lambda^{1+\frac{n}{2}}u(\tau)\|_{L^{2}}^{2}\,d\tau\nonumber\\&\quad
+C\chi_{\{3\leq n<6\}}
\left(\int_{0}^{t}
\tau^{\frac{n+2}{6-n}} e^{-\frac{(n+10)\gamma\tau}{6-n}}\,d\tau\right)^{\frac{6-n}{n+2}}
\left(\int_{0}^{t}
\|\Lambda^{1+\frac{n}{2}}u(\tau)\|_{L^{2}}^{2}\,d\tau\right)^{\frac{2(n-2)}{n+2}}
\nonumber\\
&\leq \widetilde{C}+C\chi_{\{n\geq 6\}}\int_{0}^{t}
\|\Lambda^{1+\frac{n}{2}}u(\tau)\|_{L^{2}}^{2}\,d\tau\nonumber\\&\quad
+C\chi_{\{3\leq n<6\}}
\left(\int_{0}^{t}
\|\Lambda^{1+\frac{n}{2}}u(\tau)\|_{L^{2}}^{2}\,d\tau\right)^{\frac{2(n-2)}{n+2}}
\nonumber\\&\leq \widetilde{C}.
\end{align}
We get by integrating (\ref{xcrtinm01}) in time and using (\ref{xcrtinm02}) as well as (\ref{xcrtinm04})
\begin{align}\label{xcrtinm05}&
t^{2}\|\Lambda^{\frac{1}{2}+\frac{n}{4}}\partial_{t}u(t)
 \|_{L^{2}}^{2}+t^{2}\phi(t)+\int_{0}^{t}\tau^{2}\|\sqrt{\rho}
 \partial_{\tau\tau}u(\tau)\|_{L^{2}}^{2}\,d\tau\nonumber\\&\leq  \widetilde{C} +\int_{0}^{t}
B(\tau)\tau^{2}\|\Lambda^{\frac{1}{2}+\frac{n}{4}}\partial_{\tau}u(\tau)\|_{L^{2}}^{2}\,d\tau.
\end{align}
Direct computations also yield
\begin{align}\label{xcrtinm06}
t^{2}\left|\phi(t)\right| &\leq \frac{1}{2}t^{2}\|\Lambda^{\frac{1}{2}+\frac{n}{4}}\partial_{t}u\|_{L^{2}}^{2}+
Ct^{2}\|\sqrt{\rho}\partial_{t}u\|_{L^{2}}^{2}\| \Lambda^{\frac{1}{2}+\frac{n}{4}}u\|_{L^{2}}^{2}\nonumber\\& \quad+
Ct^{2}\|\Lambda^{\frac{1}{2}+\frac{n}{4}}u\|_{L^{2}}^{\frac{2n+20}{n+2}}
\|\Lambda^{1+\frac{n}{2}}u\|_{L^{2}}^{\frac{4(n-2)}{n+2}}\nonumber\\
&= \frac{1}{2}t^{2}\|\Lambda^{\frac{1}{2}+\frac{n}{4}}\partial_{t}u\|_{L^{2}}^{2}+
Cte^{-\gamma t} (t\|\sqrt{\rho}\partial_{t}u\|_{L^{2}}^{2})(e^{\gamma t}\| \Lambda^{\frac{1}{2}+\frac{n}{4}}u\|_{L^{2}}^{2})\nonumber\\& \quad+
Ct^{\frac{8}{n+2}}e^{-\frac{(n+10)\gamma t}{n+2}}(e^{\gamma t}\|\Lambda^{\frac{1}{2}+\frac{n}{4}}u\|_{L^{2}}^{2})^{\frac{n+10}{n+2}}
(t\|\Lambda^{1+\frac{n}{2}}u\|_{L^{2}}^{2})^{\frac{2(n-2)}{n+2}}\nonumber\\
&\leq \frac{1}{2}t^{2}\|\Lambda^{\frac{1}{2}+\frac{n}{4}}\partial_{t}u\|_{L^{2}}^{2}+
\widetilde{C}.
\end{align}
Inserting (\ref{xcrtinm06}) into (\ref{xcrtinm05}) implies
\begin{eqnarray}
t^{2}\|\Lambda^{\frac{1}{2}+\frac{n}{4}}\partial_{t}u(t)
 \|_{L^{2}}^{2}+\int_{0}^{t}\tau^{2}\|\sqrt{\rho}\partial_{\tau\tau}u(\tau)\|_{L^{2}}
 ^{2}\,d\tau \leq  \widetilde{C} +\int_{0}^{t}
B(\tau)\tau^{2}\|\Lambda^{\frac{1}{2}+\frac{n}{4}}\partial_{\tau}u(\tau)\|_{L^{2}}^{2}
\,d\tau.\nonumber
\end{eqnarray}
This along with the Gronwall inequality and (\ref{xcrtinm03}) yields
\begin{eqnarray*} 
t^{2}\|\Lambda^{\frac{1}{2}+\frac{n}{4}}\partial_{t}u(t)
\|_{L^{2}}^{2} +\int_{0}^{t}{\tau^{2}\|\sqrt{\rho}\partial_{\tau\tau}u(\tau)\|_{L^{2}}^{2}\,d\tau} \leq \widetilde{C},
\end{eqnarray*}
which is (\ref{t5529b01}). With the help of (\ref{t5529b01}), we are in the position to derive the exponential decay of $\|\Lambda^{\frac{1}{2}+\frac{n}{4}}\partial_{t}u(t)
\|_{L^{2}}$. To this end, we multiply (\ref{dsvv54t227}) by $e^{\gamma t}$ to obtain
\begin{align}\label{dsvv54t228}&
\frac{d}{dt}\left(e^{\gamma t}\|\Lambda^{\frac{1}{2}+\frac{n}{4}}\partial_{t}u(t)
 \|_{L^{2}}^{2}+e^{\gamma t}\phi(t)\right) +e^{\gamma t}\|\sqrt{\rho}\partial_{tt}u\|_{L^{2}}^{2}\nonumber\\&\leq  \gamma e^{\gamma t}\|\Lambda^{\frac{1}{2}+\frac{n}{4}}\partial_{t}u(t)
 \|_{L^{2}}^{2}+\gamma e^{\gamma t}\phi(t)+A(t)e^{\gamma t}\|\Lambda^{\frac{1}{2}+\frac{n}{4}}\partial_{t}u\|_{L^{2}} +
B(t)e^{\gamma t}\|\Lambda^{\frac{1}{2}+\frac{n}{4}}\partial_{t}u\|_{L^{2}}^{2}.
\end{align}
Now integrating (\ref{dsvv54t228}) on the time interval $[1,\,t]$ yields
\begin{align}\label{dsvv54t230}&
e^{\gamma t}\|\Lambda^{\frac{1}{2}+\frac{n}{4}}\partial_{t}u(t)
\|_{L^{2}}^{2}+e^{\gamma t}\phi(t)+\int_{1}^{t}e^{\gamma \tau}\|\sqrt{\rho}\partial_{\tau\tau}u(\tau)\|_{L^{2}}^{2}\,d\tau\nonumber\\&\leq  \widetilde{C}+\gamma \int_{1}^{t}e^{\gamma \tau}\|\Lambda^{\frac{1}{2}+\frac{n}{4}}\partial_{\tau}u(\tau)
\|_{L^{2}}^{2}\,d\tau+\gamma \int_{1}^{t}e^{\gamma \tau}\phi(\tau)\,d\tau\nonumber\\&\quad+\int_{1}^{t}A(\tau)e^{\gamma \tau}\|\Lambda^{\frac{1}{2}+\frac{n}{4}}\partial_{\tau}u(\tau)\|_{L^{2}}\,d\tau +\int_{1}^{t}
B(\tau)e^{\gamma \tau}\|\Lambda^{\frac{1}{2}+\frac{n}{4}}\partial_{\tau}u(\tau)\|_{L^{2}}^{2}\,d\tau
\nonumber\\&\leq \widetilde{C}+2\gamma \int_{1}^{t}e^{\gamma \tau}\|\Lambda^{\frac{1}{2}+\frac{n}{4}}\partial_{\tau}u(\tau)
 \|_{L^{2}}^{2}\,d\tau+\gamma \int_{1}^{t}e^{\gamma \tau}\phi(\tau)\,d\tau\nonumber\\& \quad+C\int_{1}^{t}A^{2}(\tau)e^{\gamma \tau}\,d\tau +\int_{1}^{t}
B(\tau)e^{\gamma \tau}\|\Lambda^{\frac{1}{2}+\frac{n}{4}}\partial_{\tau}u(\tau)\|_{L^{2}}^{2}\,d\tau.
\end{align}
According to the estimates (\ref{addt030t003}), (\ref{521tt007}) and (\ref{xfrmu8}), it follows from (\ref{dfr56bn28}) that
\begin{eqnarray}\label{dsvv54t231}
e^{\gamma t}\left|\phi(t)\right|\leq \frac{1}{2}e^{\gamma t}\|\Lambda^{\frac{1}{2}+\frac{n}{4}}\partial_{t}u(t)\|_{L^{2}}^{2}
+\widetilde{C},
\end{eqnarray}
\begin{align}\label{dsvv54t232}
\gamma\int_{1}^{t}e^{\gamma \tau}\phi(\tau)\,d\tau\leq \widetilde{C}.
\end{align}
Appealing to the estimates (\ref{addt030t003}), (\ref{521tt007}) and (\ref{xfrmu8}) again, we can also show
\begin{eqnarray}\label{dsvv54t233}\gamma \int_{1}^{t}e^{\gamma \tau}\|\Lambda^{\frac{1}{2}+\frac{n}{4}}\partial_{\tau}u(\tau)
\|_{L^{2}}^{2}\,d\tau\leq \widetilde{C},
\end{eqnarray}
\begin{align}\label{dsvv54t234}
C\int_{1}^{t}A^{2}(\tau)e^{\gamma \tau}\,d\tau &= C\int_{1}^{t}(e^{\gamma \tau}\|\Lambda^{\frac{1}{2}+\frac{n}{4}}u(\tau)\|_{L^{2}}^{2})^{\frac{n+10}{n+2}}
(e^{\gamma \tau}\|\Lambda^{1+\frac{n}{2}}u(\tau)\|_{L^{2}}^{2})^{\frac{3n-2}{n+2}}e^{-3\gamma \tau}\,d\tau\nonumber\\
&\leq C\int_{1}^{t}e^{-3\gamma \tau}\,d\tau
\leq\widetilde{C}.
\end{align}
Inserting the above estimates (\ref{dsvv54t231})-(\ref{dsvv54t234}) into (\ref{dsvv54t230}) yields
$$e^{\gamma t}\|\Lambda^{\frac{1}{2}+\frac{n}{4}}\partial_{t}u(t)
 \|_{L^{2}}^{2}+\int_{1}^{t}e^{\gamma \tau}\|\sqrt{\rho}\partial_{\tau\tau}u(\tau)\|_{L^{2}}^{2}\,d\tau\leq \widetilde{C}+\int_{1}^{t}
B(\tau)e^{\gamma \tau}\|\Lambda^{\frac{1}{2}+\frac{n}{4}}\partial_{\tau}u(\tau)\|_{L^{2}}^{2}\,d\tau.$$
Similarly, it follows from the estimates (\ref{addt030t003}), (\ref{521tt007}) and (\ref{xfrmu8}) that
$$\int_{1}^{t}
B(\tau)\,d\tau\leq \widetilde{C}.$$
As a result, we have by the Gronwall inequality
\begin{eqnarray}
e^{\gamma t}\|\Lambda^{\frac{1}{2}+\frac{n}{4}}\partial_{t}u(t)
\|_{L^{2}}^{2} +\int_{1}^{t}{e^{\gamma \tau}\|\sqrt{\rho}\partial_{\tau\tau}u(\tau)\|_{L^{2}}^{2}\,d\tau} \leq \widetilde{C}.\nonumber
\end{eqnarray}
Consequently, we complete the proof of Lemma \ref{INDL226}.
\end{proof}

With the estimates of Lemma \ref{INDL226} at hand, we can obtain the following estimate.
\begin{lemma}\label{INDL227}
Under the assumptions of Theorem \ref{Th1}, the   solution $(\rho,u)$
of the system (\ref{INDNSE}) admits the following bound for any $t\geq1$,
\begin{eqnarray*} 
e^{\gamma t}\|\Lambda^{1+\frac{n}{2}}u(t)\|_{L^{\frac{4n}{n-2}}}^{2} +
e^{\gamma t}\|p(t)\|_{W^{1,\frac{4n}{n-2}}}^{2} \leq \widetilde{C},
\end{eqnarray*}
where $\widetilde{C}$ depends only on $\|\rho_{0}\|_{L^{\frac{2n}{n+2}}}$, $\|\rho_{0}\|_{L^{\infty}}$, $\|\nabla\rho_{0}\|_{L^{2}}$, $\|\sqrt{\rho_{0}}u_{0}\|_{L^{2}}$ and $\|\Lambda^{\frac{1}{2}+\frac{n}{4}}u_{0}\|_{L^{2}}$.
\end{lemma}

\begin{proof}
Using the regularity properties of the Stokes system (\ref{521tt005}), we get
\begin{align}
\|\Lambda^{1+\frac{n}{2}}u\|_{L^{\frac{4n}{n-2}}}+\|\nabla p\|_{L^{\frac{4n}{n-2}}}&\leq C\|\rho\partial_tu\|_{L^{\frac{4n}{n-2}}}
+C\|\rho u\cdot\nabla u\|_{L^{\frac{4n}{n-2}}}\nonumber\\
&\leq C\|{\rho}
\|_{L^{\infty}}\| \partial_tu\|_{L^{\frac{4n}{n-2}}}
+C\|{\rho}
\|_{L^{\infty}}\|u \nabla u\|_{L^{\frac{4n}{n-2}}}
\nonumber\\
&\leq C\|\Lambda^{\frac{1}{2}+\frac{n}{4}}\partial_{t}u
\|_{L^{2}}
+C\|\Lambda^{\frac{1}{2}+\frac{n}{4}}u\|_{L^{2}}\|\Lambda^{1+\frac{n}{2}}u\|_{L^{2}},
\nonumber
\end{align}
where  we have used (\ref{dfcdsadr7}) in the last line.
Recalling the  estimates obtained in the previous lemmas, we see that for any $t\geq 1$,
$$e^{\gamma t}\|\Lambda^{1+\frac{n}{2}}u(t)\|_{L^{\frac{4n}{n-2}}}^{2} +
e^{\gamma t}\|\nabla p(t)\|_{L^{\frac{4n}{n-2}}}^{2} \leq \widetilde{C}.$$
Similar argument also implies
\begin{align}
\|p\|_{L^{\frac{4n}{n-2}}}&\leq C\|\Lambda^{-1}(\rho\partial_tu\|_{L^{2}})\|_{L^{\frac{4n}{n-2}}}
+C\|\Lambda^{-1}(\rho u\cdot\nabla u)\|_{L^{\frac{4n}{n-2}}}\nonumber\\
&\leq  C\|\rho\partial_tu\|_{L^{\frac{4n}{n+2}}}
+C\|\rho u\cdot\nabla u\|_{L^{\frac{4n}{n+2}}}
\nonumber\\
&\leq C\|\rho\partial_tu\|_{L^{2}}^{\frac{4}{n+2}}\|\rho\partial_tu\|_{L^{\frac{4n}{n-2}}}
^{\frac{n-2}{n+2}}
+
C\|\rho \|_{L^{\infty}} \|u\|_{L^{\infty}}
\|\nabla u\|_{L^{\frac{4n}{n+2}}}
\nonumber\\
&\leq C(\|\sqrt{\rho}\partial_tu\|_{L^{2}}+\|\rho\partial_tu\|_{L^{\frac{4n}{n-2}}}
+\|\Lambda^{\frac{1}{2}+\frac{n}{4}}u\|_{L^{2}}^{\frac{4}{n+2}}
\|\Lambda^{1+\frac{n}{2}}u\|_{L^{2}}^{\frac{n-2}{n+2}}
\|\Lambda^{\frac{1}{2}+\frac{n}{4}}u\|_{L^{2}})
\nonumber\\
&\leq C(\|\sqrt{\rho}\partial_tu\|_{L^{2}}+\|\Lambda^{\frac{1}{2}+\frac{n}{4}}\partial_{t}u
\|_{L^{2}}
+\|\Lambda^{\frac{1}{2}+\frac{n}{4}}u\|_{L^{2}}^{\frac{n+6}{n+2}}
\|\Lambda^{1+\frac{n}{2}}u\|_{L^{2}}^{\frac{n-2}{n+2}}).\nonumber
\end{align}
Consequently, it gives that for any $t\geq1$,
$$e^{\gamma t}\|p(t)\|_{L^{\frac{4n}{n-2}}}^{2} \leq \widetilde{C}.$$
Moreover, we also deduce that for any $t\leq1$,
$$t^{2}\|\Lambda^{1+\frac{n}{2}}u(t)\|_{L^{\frac{4n}{n-2}}}^{2} +
t^{2}\|p(t)\|_{W^{1,\frac{4n}{n-2}}}^{2} \leq \widetilde{C}.$$
Hence, we obtain the desired estimates and thus complete the proof of the lemma.
\end{proof}

\vskip .1in
\subsection{The proof of Theorem \ref{Th1}}
We need the following Gronwall type inequality which will be used to guarantee the uniqueness of strong solutions (see \cite[Lemma 2.5]{lijink}).
\begin{lemma}\label{UNILemma}
Let $X_{1}(t),\,X_{2}(t),\,Y(t),\,\beta(t)$ and $\gamma(t)$ be non-negative functions. In addition, $\beta(t)$ and $t\gamma(t)$ are two integrable functions over $[0,\,T]$. Let $X_{1}(t)$ and $X_{2}(t)$ be absolutely continuous over $[0,\,T]$ and satisfy
\begin{equation*}
\left\{\begin{array}{l}
\frac{d}{dt}X_{1}(t)\leq AY^{\frac{1}{2}}(t) ,\vspace{2mm} \\
\frac{d}{dt}X_{2}(t)+Y(t)\leq \beta(t)X_{2}(t)+\gamma(t)X^{2}_{1}(t)
             \vspace{2mm}\\
X_{1}(0)=0,
\end{array}\right.
\end{equation*}
where $A$ is a positive constant. Then, the following estimates hold
\begin{eqnarray}X_{1}(t)\leq A X^{\frac{1}{2}}_{2}(0)t^{\frac{1}{2}}e^{\frac{1}{2}\int_{0}^{t}
(\beta(s)+A^{2}s\gamma(s))\,ds},\nonumber
\end{eqnarray}
\begin{eqnarray}X_{2}(t)+\int_{0}^{t}Y(s)\,ds\leq X_{2}(0)e^{\int_{0}^{t}
(\beta(s)+A^{2}s\gamma(s))\,ds}.\nonumber
\end{eqnarray}
In particular, if $X_{2}(0)=0$,   we have
\begin{eqnarray}
X_{1}(t)=X_{2}(t)=Y(t)\equiv0.\nonumber
\end{eqnarray}
\end{lemma}

We continue to prove our theorem.
The desired bounds of Theorem \ref{Th1} follow directly by putting together all the estimates of the above Lemmas \ref{INDL221}-\ref{INDL227}. Thus it remains to show the uniqueness. To this end, we make use of the following two
momentum conservation equations
$$\rho\partial_tu + \rho u\cdot\nabla u +(-\Delta)^{\frac{1}{2}+\frac{n}{4}}u+\nabla p =0,\quad \widetilde{\rho}\partial_t\widetilde{u} + \widetilde{\rho} \widetilde{u}\cdot\nabla \widetilde{u} +(-\Delta)^{\frac{1}{2}+\frac{n}{4}}\widetilde{u}+\nabla \widetilde{p} =0,$$
to obtain
 $$\rho\partial_t(u-\widetilde{u}) + \rho u\cdot\nabla (u-\widetilde{u}) +(-\Delta)^{\frac{1}{2}+\frac{n}{4}}(u-\widetilde{u})+\nabla (p-\widetilde{p}) =-(\rho-\widetilde{\rho})(\partial_t\widetilde{u}+\widetilde{u}\cdot\nabla \widetilde{u})-\rho(u-\widetilde{u})\cdot\nabla \widetilde{u}.$$
Now we deduce by multiplying the above identity by $u-\widetilde{u}$ and integrating it over $\mathbb{R}^{n}$,
\begin{eqnarray*}
 \frac{1}{2}\frac{d}{dt} \|\sqrt{\rho}(u-\widetilde{u})(t)\|_{L^{2}}^{2} +\|\Lambda^{\frac{1}{2}+\frac{n}{4}}(u-\widetilde{u})\|_{L^{2}}^{2}=J_{1}+J_{2},
\end{eqnarray*}
where
$$J_{1}:=-\int_{\mathbb{R}^{n}}(\rho-\widetilde{\rho})(\partial_t\widetilde{u}+\widetilde{u}\cdot\nabla \widetilde{u})\cdot(u-\widetilde{u})\,dx,$$
$$J_{2}:=-\int_{\mathbb{R}^{n}}\rho(u-\widetilde{u})\cdot\nabla \widetilde{u}\cdot(u-\widetilde{u})\,dx.$$
The term $J_{2}$ can be bounded by
$$J_{2}\leq C\|\nabla \widetilde{u}\|_{L^{\infty}}\|\sqrt{\rho}(u-\widetilde{u})\|_{L^{2}}^{2}.$$
For the term $J_{1}$, we have by (\ref{dfcdsadr7}),
\begin{align}
J_{1}&\leq
C\|\rho-\widetilde{\rho}\|_{L^{\frac{2n}{n+2}}}
(\|\partial_t\widetilde{u}\|_{L^{\frac{4n}{n-2}}}+\|\widetilde{u}\cdot\nabla \widetilde{u}\|_{L^{\frac{4n}{n-2}}}
)\|u-\widetilde{u}\|_{L^{\frac{4n}{n-2}}}\nonumber\\
&\leq
C\|\rho-\widetilde{\rho}\|_{L^{\frac{2n}{n+2}}}
(\|\Lambda^{\frac{1}{2}+\frac{n}{4}}\partial_t\widetilde{u}\|_{L^{2}}
+\|\Lambda^{\frac{1}{2}+\frac{n}{4}}\widetilde{u}\|_{L^{2}}
\|\Lambda^{1+\frac{n}{2}}\widetilde{u}\|_{L^{2}}
)\|\Lambda^{\frac{1}{2}+\frac{n}{4}}(u-\widetilde{u})\|_{L^{2}}
\nonumber\\
&\leq
\frac{1}{2}\|\Lambda^{\frac{1}{2}+\frac{n}{4}}(u-\widetilde{u})\|_{L^{2}}^{2}+C
(\|\Lambda^{\frac{1}{2}+\frac{n}{4}}\partial_t\widetilde{u}\|_{L^{2}}^{2}
+\|\Lambda^{\frac{1}{2}+\frac{n}{4}}\widetilde{u}\|_{L^{2}}^{2}\|\Lambda^{1+\frac{n}{2}} \widetilde{u}\|_{L^{2}}^{2}
)\|\rho-\widetilde{\rho}\|_{L^{\frac{2n}{n+2}}}^{2}.\nonumber
\end{align}
We therefore obtain
\begin{align*}
& \frac{d}{dt} \|\sqrt{\rho}(u-\widetilde{u})(t)\|_{L^{2}}^{2} +\|\Lambda^{\frac{1}{2}+\frac{n}{4}}(u-\widetilde{u})\|_{L^{2}}^{2}\nonumber\\&\leq C(\|\Lambda^{\frac{1}{2}+\frac{n}{4}}\partial_t\widetilde{u}\|_{L^{2}}^{2}
+\|\Lambda^{\frac{1}{2}+\frac{n}{4}}\widetilde{u}\|_{L^{2}}^{2}
\|\Lambda^{1+\frac{n}{2}} \widetilde{u}\|_{L^{2}}^{2}
)\|\rho-\widetilde{\rho}\|_{L^{\frac{2n}{n+2}}}^{2}
\nonumber\\& \quad+
C\|\nabla \widetilde{u}\|_{L^{\infty}}\|\sqrt{\rho}(u-\widetilde{u})\|_{L^{2}}^{2}.
\end{align*}
Using the following two density equations
$$\partial_t \rho+u\cdot\nabla\rho=0,\quad \partial_t \widetilde{\rho}+\widetilde{u}\cdot\nabla\widetilde{\rho}=0,$$
we deduce
\begin{eqnarray*} 
\partial_t (\rho-\widetilde{\rho})+u\cdot\nabla(\rho-\widetilde{\rho})=-(u-\widetilde{u})
\cdot\nabla\widetilde{\rho}.
\end{eqnarray*}
It implies that
\begin{align}
 \frac{n+2}{2n}\frac{d}{dt} \|(\rho-\widetilde{\rho})(t)\|_{L^{\frac{2n}{n+2}}}^{\frac{2n}{n+2}} &\leq C\|\rho-\widetilde{\rho}\|_{L^{\frac{2n}{n+2}}}^{\frac{2n}{n+2}-1}
 \|(u-\widetilde{u})
\cdot\nabla\widetilde{\rho}\|_{L^{\frac{2n}{n+2}}}\nonumber\\
&\leq C\|\rho-\widetilde{\rho}\|_{L^{\frac{2n}{n+2}}}^{\frac{2n}{n+2}-1}
 \|u-\widetilde{u}\|_{L^{\frac{4n}{n-2}}} \|\nabla\widetilde{\rho}\|_{L^{\frac{4n}{n+6}}}
 \nonumber\\
&\leq C\|\rho-\widetilde{\rho}\|_{L^{\frac{2n}{n+2}}}^{\frac{2n}{n+2}-1}
 \|\Lambda^{\frac{1}{2}+\frac{n}{4}}(u-\widetilde{u})\|_{L^{2}} \|\nabla\widetilde{\rho}\|_{L^{\frac{4n}{n+6}}}.\nonumber
\end{align}
We may conclude
$$\frac{d}{dt} \|(\rho-\widetilde{\rho})(t)\|_{L^{\frac{2n}{n+2}}}\leq C\|\Lambda^{\frac{1}{2}+\frac{n}{4}}(u-\widetilde{u})\|_{L^{2}} \|\nabla\widetilde{\rho}\|_{L^{\frac{4n}{n+6}}}.$$
Now let us denote
$$X_{1}(t):=\|(\rho-\widetilde{\rho})(t)\|_{L^{\frac{2n}{n+2}}},\quad
X_{2}(t):=\|\sqrt{\rho}(u-\widetilde{u})(t)\|_{L^{2}}^{2},\quad Y(t):=\|\Lambda^{\frac{1}{2}+\frac{n}{4}}(u-\widetilde{u})(t)\|_{L^{2}}^{2},$$
$$\beta(t):=C\|\nabla \widetilde{u}(t)\|_{L^{\infty}},\quad \gamma(t):=C
(\|\Lambda^{\frac{1}{2}+\frac{n}{4}}\partial_t\widetilde{u}(t)\|_{L^{2}}^{2}
+\|\Lambda^{\frac{1}{2}+\frac{n}{4}}\widetilde{u}(t)\|_{L^{2}}^{2}\|\Lambda^{1+\frac{n}{2}} \widetilde{u}(t)\|_{L^{2}}^{2}
),$$
which satisfy
\begin{equation*}
\left\{\begin{array}{l}
\frac{d}{dt}X_{1}(t)\leq AY^{\frac{1}{2}}(t) ,\vspace{2mm} \\
\frac{d}{dt}X_{2}(t)+Y(t)\leq \beta(t)X_{2}(t)+\gamma(t)X^{2}_{1}(t),
             \vspace{2mm}\\
X_{1}(0)=0.
\end{array}\right.
\end{equation*}
Recalling (\ref{addt030t003}), (\ref{521tt007}), (\ref{xctoplk243}) and (\ref{521tt009}), we know that
$$\int_{0}^{t}{\beta(\tau) \,d\tau}\leq C_{0}(t),\qquad \int_{0}^{t}{\tau\gamma(\tau) \,d\tau}\leq C_{0}(t).$$
Due to $u(x,0)=\widetilde{u}(x,0)$, we have $X_{2}(0)=0$.
Making use of the Gronwall type inequality  in  Lemma \ref{UNILemma}, we immediately have the uniqueness, namely,
$$u(x,t)=\widetilde{u}(x,t),\quad \rho(x,t)=\widetilde{\rho}(x,t).$$
This completes the proof of Theorem \ref{Th1}.
\bigskip

\appendix
\section{The case of dimension $n=2$}

As a byproduct of the approach in the proof of Theorem \ref{Th1}, we also obtain the  exponential decay-in-time of the strong solution in dimension $n=2$ provided that a damping term $u$ is added in the momentum equation. More precisely, we have the following result.
\begin{thm}\label{Th2}
Consider the following system
\begin{equation}\label{SINDNSE}
\left\{\begin{array}{l}
\partial_t \rho+{\rm div}(\rho u)
       =0 ,  \qquad   x\in \mathbb{R}^{2},\,t>0,\\
\partial_t(\rho u) + {\rm div}(\rho u\otimes u) -\Delta u+u+\nabla p =0,
              \\
\nabla\cdot u=0, \\
\rho(x,0)=\rho_{0}(x),\quad u(x,0)=u_{0}(x).
\end{array}\right.
\end{equation}
Assume that the initial data $(\rho_{0},\,u_{0})$
satisfies the following conditions:
$$0\leq\rho_{0}\in L^{1}(\mathbb{R}^{2})\cap L^{\infty}(\mathbb{R}^{2}),\quad \nabla\rho_{0}\in  L^{q}(\mathbb{R}^{2}),\quad q>2,$$
$$ \nabla\cdot u_{0}=0,\quad  u_{0}\in H^{1}(\mathbb{R}^{2}),\quad \sqrt{\rho_{0}}u_{0}\in L^{2}(\mathbb{R}^{2}).$$
Then the system (\ref{SINDNSE}) has a unique global strong
solution $(\rho,u)$ satisfying,  for any given $T>0$ and for any $0<\tau<T$,
$$0\leq\rho\in L^{\infty}(0,T; L^{1}(\mathbb{R}^{2})\cap L^{\infty}(\mathbb{R}^{2})),\quad \nabla\rho\in L^{\infty}(0,T; L^{q}(\mathbb{R}^{2}),$$
$$  u \in L^{\infty}(0,T; {H}^{1}(\mathbb{R}^{2}))\cap L^{2}(0,T; {H}^{2}(\mathbb{R}^{2}))\cap L^{\infty}(\tau,T; {\dot{W}}^{2,\,m}(\mathbb{R}^{2})),$$
$$ \sqrt{\rho}\partial_{t}u\in L^{\infty}(\tau,T; L^{2}(\mathbb{R}^{2})),\quad
\partial_{t}u\in L^{2}(\tau,T;  {H}^{1}(\mathbb{R}^{2}))\cap L^{\infty}(\tau,T;  {H}^{1}(\mathbb{R}^{2})),$$
$$\nabla p \in L^{\infty}(\tau,T; L^{2}(\mathbb{R}^{2})\cap {L}^{m}(\mathbb{R}^{2})),$$
for any $m\in(2,\,\infty)$.
Moreover, there exists some positive constant $\gamma$ depending only on $\|\rho_{0}\|_{L^{1}}$ and $\|\rho_{0}\|_{L^{\infty}}$ such that,  for all $t\geq1$,
\begin{eqnarray}\|\sqrt{\rho}\partial_{t}u(t)\|_{L^{2}}^{2}
+\|u(t)\|_{H^{2}}^{2} +\|\Delta u(t)\|_{L^{m}}^{2}+\|\partial_{t}u(t)\|_{H^{1}}^{2}+\|\nabla p(t)\|_{L^{2}\cap L^{m}}^{2}
\leq \widetilde{C}e^{-\gamma t},  \nonumber
\end{eqnarray}
where $\widetilde{C}$ depends only on $\|\rho_{0}\|_{L^{1}}$, $\|\rho_{0}\|_{L^{\infty}}$, $\|\nabla\rho_{0}\|_{L^{q}}$, $\|\sqrt{\rho_{0}}u_{0}\|_{L^{2}}$ and $\|u_{0}\|_{H^{1}}$.
\end{thm}

\begin{rem}\rm
When the damping term $u$ is absent from the system (\ref{SINDNSE}), it seems difficult to obtain the exponential decay of the strong solution as in Theorem \ref{Th2}. The key obstacle is that the classical Sobolev embedding inequality is critical in dimension $n=2$. However, if the initial density decays not too slowly at infinity, then it is  proved in \cite{LsZhong} that the corresponding system admits a unique global strong solution. Moreover, the following large-time decay rates were obtained:
$\|\nabla u(t)\|_{L^{2}}+\|\nabla^{2} u(t)\|_{L^{2}}+\|\nabla p(t)\|_{L^{2}}
\leq \widetilde{C}t^{-1}.$
\end{rem}

As the proof of Theorem \ref{Th2} can be carried out  as that of Theorem \ref{Th1} with some suitable modifications, we only give a sketch of the proof in this appendix.
First, the basic energy estimates read as follows.
\begin{lemma}\label{SINDL31}
Under the assumptions of Theorem \ref{Th2}, the   solution $(\rho,u)$
of the system (\ref{SINDNSE}) admits the following bound for any $t\geq0$,
\begin{eqnarray}\label{r58t001}
\|\rho(t)\|_{L^{1}\cap L^{\infty}}\leq
\|\rho_{0}\|_{L^{1}\cap L^{\infty}},\quad e^{\gamma t}\|\sqrt{\rho}u(t)\|_{L^{2}}^{2} + \int_{0}^{t}{e^{\gamma \tau}\|u(\tau)\|_{H^{1}}^{2} \,d\tau}\leq  \|\sqrt{\rho_{0}}u_{0}\|_{L^{2}}^{2}.
\end{eqnarray}
\end{lemma}
\begin{proof}
The first part of the estimate (\ref{r58t001}) and the non-negativeness of $\rho$ can be deduced as in Lemma \ref{INDL221}.
To show the second part of (\ref{r58t001}), we multiply equation $\rm (\ref{SINDNSE})_{2}$ by $u$ and integrate the resulting equation over $\mathbb{R}^{2}$ to get
\begin{eqnarray}
 \frac{1}{2}\frac{d}{dt} \|\sqrt{\rho}u(t)\|_{L^{2}}^{2} +\|u\|_{H^{1}}^{2}=0.\nonumber
\end{eqnarray}
Fixing $r\in (1,\,\infty)$, we see that
\begin{eqnarray}
\|\sqrt{\rho}u\|_{L^{2}}\leq C\|\sqrt{\rho}\|_{L^{2r}}\|u\|_{L^{\frac{2r}{r-1}}} \leq C\|\rho_{0}\|_{L^{1}\cap L^{\infty}}^{\frac{1}{2}}\|u\|_{H^{1}}\leq C\|u\|_{H^{1}},\nonumber
\end{eqnarray}
which is crucial for the exponential decay estimate, but different from \eqref{key1}.
It follows that
\begin{eqnarray}
 \frac{d}{dt} \|\sqrt{\rho}u(t)\|_{L^{2}}^{2} +\gamma\|\sqrt{\rho}u(t)\|_{L^{2}}^{2} +\|u\|_{H^{1}}^{2}=0.\nonumber
\end{eqnarray}
By the Gronwall inequality, one can prove
$$e^{\gamma t}\|\sqrt{\rho}u(t)\|_{L^{2}}^{2} + \int_{0}^{t}{e^{\gamma \tau} \|u(\tau)\|_{H^{1}}^{2}\,d\tau}\leq \|\sqrt{\rho_{0}}u_{0}\|_{L^{2}}^{2}.$$
This completes  the proof of Lemma \ref{SINDL31}.
\end{proof}

\begin{lemma}\label{SINDL32}
Under the assumptions of Theorem \ref{Th2}, the   solution $(\rho,u)$
of the system (\ref{SINDNSE}) admits the following bound for any $t\geq0$,
\begin{eqnarray} \label{r58t003}
e^{\gamma t}\|u(t)\|_{H^{1}}^{2} +\int_{0}^{t}{e^{\gamma \tau}(\|u(\tau)\|_{H^{2}}^{2}+\|\sqrt{\rho}
\partial_{\tau}u(\tau)\|_{L^{2}}^{2}+\|\sqrt{\rho}
\dot{u}(\tau)\|_{L^{2}}^{2})\,d\tau}\leq \widetilde{C_{1}},
\end{eqnarray}
where $\dot{u}:=\partial_{t}u+u\cdot\nabla u$ is the material derivatives of the velocity $u$, and $\widetilde{C_{1}}$ depends only on $\|\rho_{0}\|_{L^{1}}$, $\|\rho_{0}\|_{L^{\infty}}$, $\|\sqrt{\rho_{0}}u_{0}\|_{L^{2}}$ and $\| u_{0}\|_{H^{1}}$.
\end{lemma}

\begin{proof}
We first rewrite the equation $\rm (\ref{SINDNSE})_{2}$ as
\begin{eqnarray} \label{r58t004}
\rho \dot{u}=\Delta u-u-\nabla p.
\end{eqnarray}
Multiplying the equation $\rm (\ref{r58t004})$ by $\dot{u}$ and integrating it over $\mathbb{R}^{2}$   lead  to
\begin{eqnarray}\label{r58t005}
 \|\sqrt{\rho}\dot{u}\|_{L^{2}}^{2}=\int_{\mathbb{R}^{2}} \dot{u}\cdot\Delta u\,dx-\int_{\mathbb{R}^{2}} \dot{u}\cdot u\,dx-\int_{\mathbb{R}^{2}} \dot{u}\cdot\nabla p\,dx.
\end{eqnarray}
On the one hand, one has
\begin{eqnarray}
\int_{\mathbb{R}^{2}} \dot{u}\cdot\Delta u\,dx=\int_{\mathbb{R}^{2}} \partial_{t}u\cdot\Delta u\,dx+\int_{\mathbb{R}^{2}} (u\cdot\nabla u)\cdot\Delta u\,dx=-\frac{1}{2}\frac{d}{dt}\|\nabla{u}(t)\|_{L^{2}}^{2},\nonumber
\end{eqnarray}
where we have used the following fact due to $\nabla\cdot u=0$ (see \cite[(3.3)]{Wuxye} for details):
$$\int_{\mathbb{R}^{2}} (u\cdot\nabla u)\cdot\Delta u\,dx=0.$$
On the other hand, we have
$$-\int_{\mathbb{R}^{2}} \dot{u}\cdot u\,dx=-\int_{\mathbb{R}^{2}} \partial_{t}u\cdot  u\,dx-\int_{\mathbb{R}^{2}} (u\cdot\nabla u)\cdot u\,dx=-\frac{1}{2}\frac{d}{dt}\|u(t)\|_{L^{2}}^{2}.$$
Due to \cite[(3.8)]{LsZhong}, the last term in \eqref{r58t005} can be bounded by
\begin{eqnarray}\label{sdfrt001}
-\int_{\mathbb{R}^{2}} \dot{u}\cdot\nabla p\,dx
=\int_{\mathbb{R}^{2}} \partial_{j}u_{i}\partial_{i}u_{j} p\,dx
\leq C\|p\|_{\rm{BMO}}\|\partial_{j}u\cdot \nabla u_{j}\|_{\mathcal{H}^{1}}
\leq C\|\nabla p\|_{L^{2}}\|\nabla u\|_{L^{2}}^{2}.
\end{eqnarray}
We rewrite (\ref{r58t004}) as the Stokes system
\begin{equation}\label{r58t006}
\left\{\begin{array}{l}
-\Delta u+u+\nabla p =-\rho\dot{u},
             \vspace{2mm}\\
\nabla\cdot u=0.
\end{array}\right.
\end{equation}
Then, it gives
\begin{equation}\label{addr58t006}
\nabla p= (-\Delta)^{-1}\nabla\nabla\cdot(\rho\dot{u}),
\end{equation}
which yields
\begin{equation}\label{sdfrt002}\|\nabla p\|_{L^{2}}\leq C\|\rho\dot{u}\|_{L^{2}}\leq C\|\sqrt{\rho}\dot{u}\|_{L^{2}}.
\end{equation}
Combining all the above estimates implies that
\begin{eqnarray}
 \frac{d}{dt} \|u(t)\|_{H^{1}}^{2} +\|\sqrt{\rho}\dot{u}\|_{L^{2}}^{2}\leq C \|u(t)\|_{H^{1}}^{4}.\nonumber
\end{eqnarray}
This allows us to show
$$
 \frac{d}{dt} (e^{\gamma t}\|u(t)\|_{H^{1}}^{2}) +e^{\gamma t}\|\sqrt{\rho}\dot{u}(t)\|_{L^{2}}^{2}\leq \gamma e^{\gamma t}\|u(t)\|_{H^{1}}^{2}+C\|u(t)\|_{H^{1}}^{2}(e^{\gamma t} \|u(t)\|_{H^{1}}^{2}).$$
By the estimate (\ref{r58t001}) and the Gronwall inequality, we get
\begin{eqnarray*}
e^{\gamma t}\|u(t)\|_{H^{1}}^{2} +\int_{0}^{t}e^{\gamma \tau}\|\sqrt{\rho}
\dot{u}(\tau)\|_{L^{2}}^{2}\,d\tau\leq \widetilde{C_{1}}.\end{eqnarray*}
It follows from the regularity properties of Stokes system (\ref{r58t006}) that
\begin{eqnarray*}
 \int_{0}^{t}e^{\gamma \tau}\|u(\tau)\|_{H^{2}}^{2}\,d\tau\leq \int_{0}^{t}e^{\gamma \tau}\|\rho\dot{u}(\tau)\|_{L^{2}}^{2}\,d\tau\leq\int_{0}^{t}e^{\gamma \tau}\|\sqrt{\rho}\dot{u}(\tau)\|_{L^{2}}^{2}\,d\tau\leq\widetilde{C_{1}}.
\end{eqnarray*}
We can also verify,  by  \eqref{r58t001} for $\rho$ and $H^s(\R^2)\hookrightarrow L^\infty(\R^2)$ with $s>1$ for $u$,
\begin{align}\label{fgyt653n}
 \int_{0}^{t}e^{\gamma \tau}\|\sqrt{\rho}
\partial_\tau u(\tau)\|_{L^{2}}^{2}\,d\tau&\leq  \int_{0}^{t}e^{\gamma \tau}(\|\sqrt{\rho}\dot{u}\|_{L^{2}}
+\|\sqrt{\rho} u\cdot\nabla u\|_{L^{2}})\,d\tau\nonumber\\
&\leq C\int_{0}^{t}e^{\gamma \tau}( \|\sqrt{\rho}\dot{u}\|_{L^{2}}
+\|\sqrt{\rho}\|_{L^{\infty}}\|u\|_{L^{\infty}}\|\nabla u\|_{L^{2}})\,d\tau
\nonumber\\
&\leq
C \int_{0}^{t}e^{\gamma \tau}(\|\sqrt{\rho}\dot{u}\|_{L^{2}}
+\|u\|_{H^{2}}^{2})\,d\tau\leq \widetilde{C_{1}}.
\end{align}
We thus complete the proof of the lemma.
\end{proof}

\vskip .1in

\begin{lemma}\label{SINDL33}
Under the assumptions of Theorem \ref{Th2}, the   solution $(\rho,u)$
of the system (\ref{SINDNSE}) admits the following bound for any $t\geq0$,
\begin{eqnarray} \label{rxcr5yn0}
t\|\nabla p(t)\|_{L^{2}}^{2}+t\|\sqrt{\rho}\dot{u}(t)\|_{L^{2}}^{2}+\int_{0}^{t}{\tau\|\dot{u}(\tau)\|_{H^{1}}^{2}\,d\tau}\leq \widetilde{C_{1}},
\end{eqnarray}
moreover, for any $t_{0}>0$ and any $t\geq t_{0}$, the following holds true
\begin{align} \label{r58t010}
e^{\gamma t}\|\sqrt{\rho}\dot{u}(t)\|_{L^{2}}^{2}+e^{\gamma t}\|\sqrt{\rho}
\partial_{t}u(t)\|_{L^{2}}^{2}&+ e^{\gamma t}\|u(t)\|_{H^{2}}^{2}+e^{\gamma t}\|\nabla p(t)\|_{L^{2}}^{2}+\int_{t_{0}}^{t}{e^{\gamma \tau}\|\dot{u}(\tau)\|_{H^{1}}^{2}\,d\tau}\nonumber\\&\quad+
\int_{t_{0}}^{t}{e^{\gamma \tau}\|\partial_{t}u(\tau)\|_{H^{1}}^{2}\,d\tau}\leq \frac{e^{\gamma t_{0}}}{t_{0}}\widetilde{C_{1}}:=C_{t_{0}},
\end{align}
where $\widetilde{C_{1}}$ depends only on $\|\rho_{0}\|_{L^{1}}$, $\|\rho_{0}\|_{L^{\infty}}$, $\|\sqrt{\rho_{0}}u_{0}\|_{L^{2}}$ and $\| u_{0}\|_{H^{1}}$.
\end{lemma}

\begin{proof}
According to the proof of \cite[Lemma 3.3]{LsZhong}, we have
\begin{eqnarray*}
 \frac{d}{dt}\left(\|\sqrt{\rho}\dot{u}(t)\|_{L^{2}}^{2}+\varphi(t)\right)
 +\|\dot{u}(t)\|_{H^{1}}^{2}
 \leq C(\|\nabla u\|_{L^{4}}^{4}+\|p\|_{L^{4}}^{4}),
\end{eqnarray*}
where
$\varphi(t):=-\int_{\mathbb{R}^{2}}p\partial_{j}u_{i}\partial_{i}u_{j}\,dx.$
The following estimate is an easy consequence of (\ref{sdfrt001}) and  (\ref{sdfrt002})
\begin{eqnarray}\label{xdedg68j}
|\varphi(t)|\leq C\|\sqrt{\rho}\dot{u}\|_{L^{2}}\|\nabla u\|_{L^{2}}^{2}\leq \frac{1}{2}\|\sqrt{\rho}\dot{u}\|_{L^{2}}^{2}+C\|\nabla u\|_{L^{2}}^{4}.
\end{eqnarray}
According to (\ref{r58t006}) and (\ref{addr58t006}), we have
\begin{eqnarray}\label{r58t013}
u=-(\mathbb{I}-\Delta)^{-1}\left(\rho\dot{u}+(-\Delta)^{-1}\nabla\nabla\cdot(\rho\dot{u})\right),
\end{eqnarray}
where $\mathbb{I}$ is an identity operator. Therefore, one concludes
$$\|\nabla u\|_{L^{4}}^{4}+\|p\|_{L^{4}}^{4}\leq C(\|\Delta u\|_{L^{\frac{4}{3}}}^{4}+\|\nabla p\|_{L^{\frac{4}{3}}}^{4})\leq C\|\rho\dot{u}\|_{L^{\frac{4}{3}}}^{4}\leq C\|\rho\|_{L^{2}}^{2}\|\sqrt{\rho}\dot{u}\|_{L^{2}}^{4},$$
which along with (\ref{r58t013}) gives
\begin{eqnarray}\label{r58t014}
 \frac{d}{dt}\left(\|\sqrt{\rho}\dot{u}(t)\|_{L^{2}}^{2}+\varphi(t)\right)
 +\|\dot{u}(t)\|_{H^{1}}^{2}
 \leq C\|\sqrt{\rho}\dot{u}\|_{L^{2}}^{4},
\end{eqnarray}
which then  implies
\begin{eqnarray}\label{xdfeswq12}
 \frac{d}{dt}\left(t\|\sqrt{\rho}\dot{u}(t)\|_{L^{2}}^{2}+t\varphi(t)\right)
 +t\|\dot{u}(t)\|_{H^{1}}^{2}
 \leq\|\sqrt{\rho}\dot{u}(t)\|_{L^{2}}^{2}+ Ct\|\sqrt{\rho}\dot{u}\|_{L^{2}}^{4}.
\end{eqnarray}
By (\ref{xdedg68j}) and the Gronwall inequality, one has
\begin{eqnarray}\label{polkq232}
t\|\sqrt{\rho}\dot{u}(t)\|_{L^{2}}^{2}+\int_{0}^{t}{\tau\|\dot{u}(\tau)\|_{H^{1}}^{2}\,d\tau}\leq \widetilde{C_{1}}.\end{eqnarray}
Multiplying (\ref{r58t014}) by $e^{\gamma t}$ yields
\begin{align}\label{r58t018}
\frac{d}{dt}\left(e^{\gamma t}\|\sqrt{\rho}\dot{u}(t)\|_{L^{2}}^{2}+e^{\gamma t}\varphi(t)\right) +e^{\gamma t}\|\dot{u}(t)\|_{H^{1}}^{2}&\leq  \gamma e^{\gamma t}\|\sqrt{\rho}\dot{u}(t)
 \|_{L^{2}}^{2}+\gamma e^{\gamma t}\varphi(t)\nonumber\\&\quad+Ce^{\gamma t}\|\sqrt{\rho}\dot{u}\|_{L^{2}}^{4}.
\end{align}
We thus have by integrating (\ref{r58t018}) in time and using (\ref{r58t003}), (\ref{xdedg68j}) as well as (\ref{polkq232}),
\begin{align} &
e^{\gamma t}\|\sqrt{\rho}\dot{u}(t)\|_{L^{2}}^{2}+\int_{t_{0}}^{t}e^{\gamma \tau}\|\dot{u}(\tau)\|_{H^{1}}^{2}\,d\tau\nonumber\\&\leq  C_{t_{0}}+e^{\gamma t}|\varphi(t)|+\gamma \int_{t_{0}}^{t}e^{\gamma \tau}\|\sqrt{\rho}\dot{u}(\tau)\|_{L^{2}}^{2}\,d\tau+\gamma \int_{t_{0}}^{t}e^{\gamma \tau}\varphi(\tau)\,d\tau +\int_{t_{0}}^{t}e^{\gamma \tau}\|\sqrt{\rho}\dot{u}(\tau)\|_{L^{2}}^{4}\,d\tau
\nonumber\\&\leq  C_{t_{0}}+\frac{1}{2}e^{\gamma t}\|\sqrt{\rho}\dot{u}(t)\|_{L^{2}}^{2}+ C \int_{t_{0}}^{t}\|\sqrt{\rho}\dot{u}(\tau)\|_{L^{2}}^{2}(e^{\gamma \tau}\|\sqrt{\rho}\dot{u}(\tau)\|_{L^{2}}^{2})\,d\tau.\nonumber
\end{align}
This implies
\begin{eqnarray} &&
e^{\gamma t}\|\sqrt{\rho}\dot{u}(t)\|_{L^{2}}^{2}+\int_{t_{0}}^{t}e^{\gamma \tau}\|\dot{u}(\tau)\|_{H^{1}}^{2}\,d\tau\leq
C_{t_{0}}+ C\int_{t_{0}}^{t}\|\sqrt{\rho}\dot{u}(\tau)\|_{L^{2}}^{2}(e^{\gamma \tau}\|\sqrt{\rho}\dot{u}(\tau)\|_{L^{2}}^{2})\,d\tau.\nonumber
\end{eqnarray}
By means of (\ref{r58t003}) again and the Gronwall inequality, one obtains
\begin{eqnarray} &&
e^{\gamma t}\|\sqrt{\rho}\dot{u}(t)\|_{L^{2}}^{2}+\int_{t_{0}}^{t}e^{\gamma \tau}\|\dot{u}(\tau)\|_{H^{1}}^{2}\,d\tau\leq
C_{t_{0}}.\nonumber
\end{eqnarray}
Thanks to (\ref{sdfrt002}) and (\ref{r58t013}), we get
\begin{eqnarray} &&
e^{\gamma t}\|\nabla p(t)\|_{L^{2}}^{2}+e^{\gamma t}\|u(t)\|_{H^{2}}^{2}\leq Ce^{\gamma t}\|\sqrt{\rho}\dot{u}(t)\|_{L^{2}}^{2}\leq
C_{t_{0}}.\nonumber
\end{eqnarray}
The following estimate follows immediately from (\ref{fgyt653n})
\begin{eqnarray}
e^{\gamma t}\|\sqrt{\rho}\partial_{t}u(t)\|_{L^{2}}^{2}\leq e^{\gamma t}\|\sqrt{\rho}\dot{u}(t)\|_{L^{2}}^{2}+e^{\gamma t}\|u(t)\|_{H^{2}}^{2}
\leq
 C_{t_{0}}.\nonumber
\end{eqnarray}
Finally, we have
\begin{align}
\int_{t_{0}}^{t}{e^{\gamma \tau}\|\partial_{t}u(\tau)\|_{H^{1}}^{2}\,d\tau}&\leq \int_{t_{0}}^{t}{e^{\gamma \tau}\|\dot{u}(\tau)\|_{H^{1}}^{2}\,d\tau}+\int_{t_{0}}^{t}{e^{\gamma \tau}\|(u\cdot\nabla u)(\tau)\|_{H^{1}}^{2}\,d\tau}\nonumber\\
&\leq C\int_{t_{0}}^{t}{e^{\gamma \tau}\|\dot{u}(\tau)\|_{H^{1}}^{2}\,d\tau}+C\int_{t_{0}}^{t}{e^{\gamma \tau}\|u(\tau)\|_{H^{2}}^{4}\,d\tau}\leq C_{t_{0}},\nonumber
\end{align}
from the estimate
\begin{eqnarray}\label{adrtyew1}
\|u\cdot\nabla u\|_{H^{1}}\leq C\|u u\|_{H^{2}}\leq C\|u\|_{H^{2}}^{2}. \end{eqnarray}
Therefore, we complete the proof of the lemma.
\end{proof}

\begin{lemma}\label{SINDL34}
Under the assumptions of Theorem \ref{Th2}, the   solution $(\rho,u)$
of the system (\ref{SINDNSE}) admits the following bounds for any $t\geq0$,
\begin{eqnarray} \label{r58t021}
\int_{0}^{t}{\|\nabla u(\tau)
\|_{L^{\infty}} \,d\tau} \leq \widetilde{C_{1}},\qquad \|\nabla\rho(t)\|_{L^{q}} \leq \widetilde{C_{1}}\|\nabla\rho_{0}\|_{L^{q}},
\end{eqnarray}
where $\widetilde{C_{1}}$ depends only on $\|\rho_{0}\|_{L^{1}}$, $\|\rho_{0}\|_{L^{\infty}}$, $\|\sqrt{\rho_{0}}u_{0}\|_{L^{2}}$ and $\| u_{0}\|_{H^{1}}$.
\end{lemma}

\begin{proof}
For any $2<p<\infty$, we get
$\|\rho\dot{u}\|_{L^{p}}\leq C\|\rho\|_{L^{\infty}}
\|\dot{u}\|_{H^{1}}\leq C
\|\dot{u}\|_{H^{1}}.
$
\   Applying the $L^{p}$-estimate to (\ref{r58t006}) gives
\begin{eqnarray}
\|\nabla u\|_{L^{\infty}}\leq C\|\nabla u\|_{L^{2}}^{\frac{p-2}{2p-2}}\|\Delta u\|_{L^{p}}^{\frac{p}{2p-2}} \leq C\|\nabla u\|_{L^{2}}^{\frac{p-2}{2p-2}}\|\rho\dot{u}\|_{L^{p}}^{\frac{p}{2p-2}} \leq C\|\nabla u\|_{L^{2}}^{\frac{p-2}{2p-2}}\| \dot{u}\|_{H^{1}}^{\frac{p}{2p-2}}. \nonumber
\end{eqnarray}
According to (\ref{r58t003}), (\ref{rxcr5yn0}) and (\ref{r58t010}), we  obtain the first estimate of (\ref{r58t021}).
The second estimate of (\ref{r58t021}) is a direct consequence of the first estimate.  The proof of the lemma is completed.
\end{proof}

\begin{lemma}\label{SINDL35}
Under the assumptions of Theorem \ref{Th2}, the   solution $(\rho,u)$
of the system (\ref{SINDNSE}) admits the following bound for any $m\in (2,\,\infty)$,
\begin{eqnarray*} 
e^{\gamma t}\|\partial_{t}u(t)
\|_{H^{1}}^{2}+e^{\gamma t}\|\Delta u(t)\|_{L^{m}}^{2}+e^{\gamma t}\|\nabla p(t)\|_{ L^{m}}^{2}+\int_{1}^{t}{e^{\gamma \tau}\|\sqrt{\rho}\partial_{\tau\tau}u(\tau)\|_{L^{2}}^{2}\,d\tau} \leq \widetilde{C},
\end{eqnarray*}
where $\widetilde{C}$ depends only on $\|\rho_{0}\|_{L^{1}}$, $\|\rho_{0}\|_{L^{\infty}}$, $\|\nabla\rho_{0}\|_{L^{q}}$, $\|\sqrt{\rho_{0}}u_{0}\|_{L^{2}}$ and $\|u_{0}\|_{H^{1}}$.
\end{lemma}

\begin{proof}
The proof can be performed by modifying that proof of Lemma \ref{INDL226}.
We first have by (\ref{dsvv54t226}) that
\begin{eqnarray*}
 \frac{1}{2}\frac{d}{dt}\|\partial_{t}u(t)
 \|_{H^{1}}^{2} +\|\sqrt{\rho}\partial_{tt}u\|_{L^{2}}^{2}=H_{1}+H_{2}+H_{3}+H_{4}.
\end{eqnarray*}
The $H_{3}$ and $H_{4}$ can be easily bounded by
\begin{align*}
|H_{3}|+|H_{4}|\leq& C\|\sqrt{\rho}\|_{L^{\infty}}\|\sqrt{\rho}\partial_{tt}u\|_{L^{2}}
(\|\partial_{t}u\|_{L^{4}}\|\nabla u\|_{L^{4}}+\|\nabla\partial_{t}u\|_{L^{2}}\|u\|_{L^{\infty}}) \\
\leq& C\|\sqrt{\rho_{0}}\|_{L^{\infty}}\|\sqrt{\rho}\partial_{tt}u\|_{L^{2}}
\|\partial_{t}u\|_{H^{1}}\|u\|_{H^{2}}
\leq\frac{1}{16}\|\sqrt{\rho}\partial_{tt}u\|_{L^{2}}^{2}+
C\|u\|_{H^{2}}^{2}\|\partial_{t}u\|_{H^{1}}^{2}.
\nonumber
\end{align*}
Recalling (\ref{w34cvty76}), we thus obtain
\begin{align*}
H_{1}
=& -\frac{1}{2}\frac{d}{dt}\int_{\mathbb{R}^{2}}\partial_{t}\rho |\partial_{t}u|^{2}\,dx+\int_{\mathbb{R}^{2}}\partial_{t}\rho u_{i} \partial_{t}u\cdot\partial_{t}\partial_{i}u\,dx+\int_{\mathbb{R}^{2}}\rho \partial_{t}u_{i} \partial_{t}u\cdot\partial_{t}\partial_{i}u\,dx \nonumber\\
\leq&-\frac{1}{2}\frac{d}{dt}\int_{\mathbb{R}^{2}}\partial_{t}\rho |\partial_{t}u|^{2}\,dx+C\|u\cdot\nabla\rho\|_{L^{q}}\|u\|_{L^{\infty}}
\|\partial_{t}u\|_{L^{\frac{2q}{q-2}}} \|\partial_{t}\nabla u\|_{L^{2}}\nonumber\\
&+C\|\rho\|_{L^{\infty}}
\|\partial_{t}u\|_{L^{4}}^{2} \|\partial_{t}\nabla u\|_{L^{2}}\nonumber\\
\leq&-\frac{1}{2}\frac{d}{dt}\int_{\mathbb{R}^{2}}\partial_{t}\rho |\partial_{t}u|^{2}\,dx+C\|\nabla\rho\|_{L^{q}}\|u\|_{L^{\infty}}^{2}
\|\partial_{t}u\|_{L^{\frac{2q}{q-2}}} \|\partial_{t}\nabla u\|_{L^{2}}\nonumber\\
&+C\|\rho\|_{L^{\infty}}
\|\partial_{t}u\|_{L^{4}}^{2} \|\partial_{t}\nabla u\|_{L^{2}}
\nonumber\\
\leq&-\frac{1}{2}\frac{d}{dt}\int_{\mathbb{R}^{2}}\partial_{t}\rho |\partial_{t}u|^{2}\,dx+C\|u\|_{H^{2}}^{2}
\|\partial_{t}u\|_{H^{1}} \|\partial_{t}\nabla u\|_{L^{2}}+C
\|\partial_{t}u\|_{H^{1}}^{2} \|\partial_{t}\nabla u\|_{L^{2}}\nonumber\\
\leq&-\frac{1}{2}\frac{d}{dt}\int_{\mathbb{R}^{2}}\partial_{t}\rho |\partial_{t}u|^{2}\,dx+C\|u\|_{H^{2}}^{2}
\|\partial_{t}u\|_{H^{1}}^{2}+C
\|\partial_{t}u\|_{H^{1}} \|\partial_{t} u\|_{H^{1}}^{2}.\nonumber
\end{align*}
In view of (\ref{w34cvty78}), one has
\begin{eqnarray}
H_{2}=-\frac{d}{dt}\int_{\mathbb{R}^{2}}\partial_{t}\rho u\cdot\nabla u\cdot \partial_{t}u\,dx+H_{21}+H_{22}+H_{23}.\nonumber
\end{eqnarray}
Now we may deduce that
\begin{align}
H_{21}\leq& C\|\rho\|_{L^{\infty}}\|\partial_{t}u\|_{L^{4}}(\|\nabla u\|_{L^{4}}^{2}\|\partial_{t}u\|_{L^{4}}+\|u\|_{L^{\infty}}\|\Delta u\|_{L^{2}}\|\partial_{t}u\|_{L^{4}}\nonumber\\& +\|u\|_{L^{\infty}}\|\nabla u\|_{L^{4}}\|\nabla \partial_{t}u\|_{L^{2}})
\nonumber\\ \leq&
C\|\rho_{0}\|_{L^{\infty}}\|\partial_{t}u\|_{H^{1}}(\| u\|_{H^{2}}^{2}\|\partial_{t}u\|_{H^{1}}+\|u\|_{H^{2}}^{2}\|\partial_{t}u\|_{H^{1}}+\|u\|_{H^{2}}^{2}\|\nabla \partial_{t}u\|_{L^{2}})\nonumber\\ \leq&
C\|u\|_{H^{2}}^{2}\|\partial_{t}u\|_{H^{1}}^{2},\nonumber
\end{align}
\begin{align}
H_{22}  \leq&C\|u\cdot\nabla\rho\|_{L^{q}}(\|u\|_{L^{\infty}}^{2}\|\nabla u\|_{L^{\frac{2q}{q-2}}} \|\nabla\partial_{t}u\|_{L^{2}}+
\|u\|_{L^{\infty}}\|\nabla u\|_{L^{4}}^{2} \|\partial_{t}u\|_{L^{\frac{2q}{q-2}}}
\nonumber\\& +
\|u\|_{L^{\infty}}^{2}\|\Delta u\|_{L^{2}}  \|\partial_{t}u\|_{L^{\frac{2q}{q-2}}})
\nonumber\\ \leq&C\|\nabla\rho\|_{L^{q}}(\|u\|_{L^{\infty}}^{3}\|\nabla u\|_{L^{\frac{2q}{q-2}}} \|\nabla\partial_{t}u\|_{L^{2}}+
\|u\|_{L^{\infty}}^{2}\|\nabla u\|_{L^{4}}^{2} \|\partial_{t}u\|_{L^{\frac{2q}{q-2}}}
\nonumber\\& +
\|u\|_{L^{\infty}}^{3}\|\Delta u\|_{L^{2}}  \|\partial_{t}u\|_{L^{\frac{2q}{q-2}}})
\nonumber\\ \leq&C \|u\|_{H^{2}}^{4}
\|\partial_{t}u\|_{H^{1}},\nonumber
\end{align}
\begin{align}
H_{23}  \leq&C\|u\cdot\nabla\rho\|_{L^{q}}(\|\partial_{t}u\|_{L^{4}}^{2}\|\nabla u\|_{L^{\frac{2q}{q-2}}} +
\|u\|_{L^{\infty}}\|\nabla u\|_{L^{\frac{2q}{q-2}}} \|\nabla\partial_{t}u\|_{L^{2}})
\nonumber\\ \leq&C\|\nabla\rho\|_{L^{q}}(\|u\|_{L^{\infty}}\|\partial_{t}u\|_{L^{4}}^{2}\|\nabla u\|_{L^{\frac{2q}{q-2}}} +
\|u\|_{L^{\infty}}^{2}\|\nabla u\|_{L^{\frac{2q}{q-2}}} \|\nabla\partial_{t}u\|_{L^{2}})
\nonumber\\ \leq&C(\|u\|_{H^{2}}^{2}\|\partial_{t}u\|_{H^{1}}^{2}+
\|u\|_{H^{2}}^{3} \|\partial_{t}u\|_{H^{1}}).\nonumber
\end{align}
One thus deduces
$$
H_{2}\leq -\frac{d}{dt}\int_{\mathbb{R}^{2}}\partial_{t}\rho u\cdot\nabla u\cdot \partial_{t}u\,dx+C(\|u\|_{H^{2}}^{2}\|\partial_{t}u\|_{H^{1}}^{2}+\|u\|_{H^{2}}^{4}
\|\partial_{t}u\|_{H^{1}}+
\|u\|_{H^{2}}^{3} \|\partial_{t}u\|_{H^{1}}).$$
Putting all the above estimates together implies that
\begin{eqnarray}\label{r58t026}
\frac{d}{dt}\left(\|\partial_{t}u(t)
 \|_{H^{1}}^{2}+\phi(t)\right) +\|\sqrt{\rho}\partial_{tt}u(t)\|_{L^{2}}^{2}\leq C
\|\partial_{t}u\|_{H^{1}} \|\partial_{t} u\|_{H^{1}}^{2}+R(t),
\end{eqnarray}
where
$$R(t):=C(\|u(t)\|_{H^{2}}^{2}\|\partial_{t}u(t)\|_{H^{1}}^{2}+\|u(t)\|_{H^{2}}^{4}
\|\partial_{t}u(t)\|_{H^{1}}+
\|u(t)\|_{H^{2}}^{3} \|\partial_{t}u(t)\|_{H^{1}}),$$
$$\phi(t):=-\frac{1}{2} \int_{\mathbb{R}^{2}}\partial_{t}\rho |\partial_{t}u|^{2}\,dx- \int_{\mathbb{R}^{2}}\partial_{t}\rho u\cdot\nabla u\cdot \partial_{t}u\,dx.$$
By the embedding inequality, we also get
\begin{align}
\left|\phi(t)\right| =&\left|- \int_{\mathbb{R}^{2}} \rho u_{i}\partial_{t}u\cdot\partial_{t}\partial_{i}u\,dx- \int_{\mathbb{R}^{2}}\partial_{t}\rho u\cdot\nabla u\cdot \partial_{t}u\,dx\right|
\nonumber\\
 \leq&C\|\sqrt{{\rho}}\|_{L^{\infty}}\|\sqrt{\rho}\partial_{t}u\|_{L^{2}}\| u\|_{L^{\infty}} \|\nabla\partial_{t}u\|_{L^{2}}
+C\|u\cdot\nabla\rho\|_{L^{q}}\|u\|_{L^{\infty}}\|\nabla u\|_{L^{4}} \|\partial_{t}u\|_{L^{4}}
\nonumber\\
 \leq&C  \| u\|_{H^{2}} \|\sqrt{\rho}\partial_{t}u\|_{L^{2}}\|\partial_{t}u\|_{H^{1}}
+C \| u\|_{H^{2}}^{3} \|\partial_{t}u\|_{H^{1}}
\nonumber\\
 \leq&\frac{1}{2}\|\partial_{t}u(t)\|_{H^{1}}^{2}+C\| u\|_{H^{2}}^{6}+\| u\|_{H^{2}} ^{2} \|\sqrt{\rho}\partial_{t}u\|_{L^{2}}^{2}.\nonumber
\end{align}
This leads to
\begin{eqnarray}\label{r58t027}
\left|\phi(t)\right|\leq\frac{1}{2}\|\partial_{t}u(t)\|_{H^{1}}^{2}+C\| u(t)\|_{H^{2}}^{6}+\| u(t)\|_{H^{2}} ^{2} \|\sqrt{\rho}\partial_{t}u(t)\|_{L^{2}}^{2}.
\end{eqnarray}
Now we multiply (\ref{r58t026}) by $e^{\gamma t}$ to obtain
\begin{align}\label{r58t027} &
\frac{d}{dt}\left(e^{\gamma t}\|\partial_{t}u(t)
 \|_{H^{1}}^{2}+e^{\gamma t}\phi(t)\right) +e^{\gamma t}\|\sqrt{\rho}\partial_{tt}u\|_{L^{2}}^{2} \nonumber\\ &\leq \gamma e^{\gamma t}\|\partial_{t}u(t)
 \|_{H^{1}}^{2}+\gamma e^{\gamma t}\phi(t) +Ce^{\gamma t}
\|\partial_{t}u\|_{H^{1}} \|\partial_{t} u\|_{H^{1}}^{2}+e^{\gamma t}R(t).
\end{align}
For any $t\geq1$, by (\ref{rxcr5yn0}) and (\ref{r58t010}), there exists $\sigma\in (\frac{1}{2},\,1)$ such that
$$e^{\gamma \sigma}\|\partial_{t}u(\sigma)
 \|_{H^{1}}^{2}+e^{\gamma \sigma}\phi(\sigma):=\widetilde{C}.$$
It follows from (\ref{r58t010}) again
$$\int_{\frac{1}{2}}^{t}e^{\gamma \tau}\|\partial_{\tau}u(\tau)
 \|_{H^{1}}^{2}\,d\tau+\int_{\frac{1}{2}}^{t}e^{\gamma \tau} \phi(\tau)\,d\tau +\int_{\frac{1}{2}}^{t}e^{\gamma\tau}R(\tau)\,d\tau\leq \widetilde{C}.$$
Noticing the above estimate, we integrate (\ref{r58t027}) on the time interval $[\sigma,\,t]$ to show
\begin{align}\label{r58t029}
& e^{\gamma t}\|\partial_{t}u(t)
 \|_{H^{1}}^{2}+e^{\gamma t}\phi(t)+\int_{\sigma}^{t}e^{\gamma \tau}\|\sqrt{\rho}\partial_{\tau\tau}u(\tau)\|_{L^{2}}^{2}\,d\tau
 \nonumber\\
 &\leq  \widetilde{C}+C\int_{\sigma}^{t}
\|\partial_{\tau}u(\tau)\|_{H^{1}}(e^{\gamma \tau}\|\partial_{\tau} u(\tau)\|_{H^{1}}^{2})\,d\tau.
\end{align}
From  (\ref{r58t027}) and (\ref{r58t010}), it follows that for any $t\geq \frac{1}{2}$,
\begin{eqnarray}\label{r58t030}
e^{\gamma t}\left|\phi(t)\right|\leq \frac{1}{2}e^{\gamma t}\|\partial_{t}u(t)\|_{H^{1}}^{2}
+\widetilde{C}.
\end{eqnarray}
Combining (\ref{r58t029}) and (\ref{r58t030}) ensures
$$
e^{\gamma t}\|\partial_{t}u(t)
 \|_{H^{1}}^{2}+\int_{\sigma}^{t}e^{\gamma \tau}\|\sqrt{\rho}\partial_{\tau\tau}u(\tau)\|_{L^{2}}^{2}\,d\tau \leq  \widetilde{C}+C\int_{\sigma}^{t}
\|\partial_{\tau}u(\tau)\|_{H^{1}}(e^{\gamma \tau}\|\partial_{\tau} u(\tau)\|_{H^{1}}^{2})\,d\tau.
$$
The Gronwall inequality and (\ref{r58t010}) allow us to conclude that for any $t\geq \sigma$,
\begin{eqnarray}
e^{\gamma t}\|\partial_{t}u(t)
 \|_{H^{1}}^{2}+\int_{\sigma}^{t}e^{\gamma \tau}\|\sqrt{\rho}\partial_{\tau\tau}u(\tau)\|_{L^{2}}^{2}\,d\tau \leq \widetilde{C}.\nonumber
\end{eqnarray}
Since $\sigma\leq 1$, we further have for any $t\geq1$,
\begin{eqnarray}\label{ccfr56n8}
e^{\gamma t}\|\partial_{t}u(t)
 \|_{H^{1}}^{2}+\int_{1}^{t}e^{\gamma \tau}\|\sqrt{\rho}\partial_{\tau\tau}u(\tau)\|_{L^{2}}^{2}\,d\tau \leq \widetilde{C}.
\end{eqnarray}
By (\ref{adrtyew1}), (\ref{addr58t006}), (\ref{r58t013}) and (\ref{ccfr56n8}), we derive that for any $t\geq 1$,
\begin{align}
e^{\gamma t}\|\Delta u(t)\|_{L^{m}}^{2}+e^{\gamma t}\|\nabla p(t)\|_{L^{m}}^{2}
&\leq  Ce^{\gamma t}(\|{\rho}\partial_{t}u(t)\|_{L^{m}}+\|{\rho}u\cdot\nabla u(t)\|_{L^{m}})^{2}\nonumber\\
&\leq  Ce^{\gamma t}(\|\partial_{t}u(t)\|_{H^{1}}+\|u\cdot\nabla u(t)\|_{H^{1}})^{2}\nonumber\\
&\leq  Ce^{\gamma t}(\|\partial_{t}u(t)\|_{H^{1}}^{2}+\|u(t)\|_{H^{2}}^{4})\leq \widetilde{C}.\nonumber
\end{align}
This finishes the proof of Lemma \ref{SINDL35}.
\end{proof}

Finally, Theorem \ref{Th2} follows immediately  from Lemmas \ref{SINDL31}-\ref{SINDL35}.

\bigskip

\section*{Acknowledgments}
D. Wang's research was supported in part by the National Science Foundation under grant DMS-1613213.
Z. Ye was supported by the Foundation of Jiangsu Normal University (No. 16XLR029), the Natural Science Foundation of Jiangsu
Province (No. BK20170224) and the National Natural Science Foundation of China (No. 11701232).

\end{document}